\pdfoutput=1
\documentclass[a4paper,reqno,11pt]{article}
\usepackage[utf8]{inputenc}
\usepackage[T1]{fontenc}
\usepackage{lmodern}
\usepackage{libertine}
\usepackage[english]{babel}
\usepackage[hmargin=2.45cm,vmargin=2.5cm,marginparwidth=2cm]{geometry}
\usepackage{amsmath,amssymb,amsthm,amsfonts}
\usepackage{mathtools}
\usepackage[colorlinks=true,citecolor=blue,linkcolor=magenta]{hyperref}
\usepackage{enumitem}
\usepackage{bm}
\usepackage[todonotes={textsize=scriptsize}]{changes}

\newcounter{count}

\newtheorem{dfn}[count]{Definition}
\newtheorem{thm}[count]{Theorem}
\newtheorem{rmk}[count]{Remark}
\newtheorem{lem}[count]{Lemma}
\newtheorem{prop}[count]{Proposition}
\newtheorem{cor}[count]{Corollary}

\newcommand{\1}{\bm{1}}
\newcommand{\ii}{\mathbf{i}}
\newcommand{\BL}{\mathrm{BL}}
\newcommand{\PP}{\mathbb{P}}
\newcommand{\EE}{\mathbb{E}}
\newcommand{\RR}{\mathbb{R}}
\newcommand{\NN}{\mathbb{N}}
\newcommand{\dd}{\mathrm{d}}

\newcommand{\authora}[1]{\gdef\authora{#1}}
\newcommand{\addressa}[1]{\gdef\addressa{#1}}
\newcommand{\emaila}[1]{\gdef\emaila{#1}}

\newcommand{\authorb}[1]{\gdef\authorb{#1}}
\newcommand{\addressb}[1]{\gdef\addressb{#1}}
\newcommand{\emailb}[1]{\gdef\emailb{#1}}

\makeatletter
\newcommand{\@endstuff}{
  \bigskip\small\scshape
  \noindent\authora\\
  \addressa\\
  E-mail address: \textnormal{\href{mailto:\emaila}{\texttt{\emaila}}}\\
  \medskip\\
  \authorb\\
  \addressb\\
  E-mail address: \textnormal{\href{mailto:\emailb}{\texttt{\emailb}}}}
\AtEndDocument{\@endstuff}
\makeatother

\authora{Joaquín Fontbona}
\addressa{Departamento de Ingenier\'ia Matem\'atica (DIM) and Centro de Modelamiento Matemático (CMM, UMI-CNRS 2807), 
  Universidad de Chile, 
  Santiago, Chile.}
\emaila{fontbona@dim.uchile.cl}
 
\authorb{Felipe Muñoz-Hernández}
\addressb{Departamento de Ingenier\'ia Matem\'atica (DIM), 
  Universidad de Chile, 
  Santiago, Chile.\\
  Centre de Mathématiques Appliquées (CMAP), École Polytechnique, CNRS, Palaiseau, France.}
\emailb{fmunozh@dim.uchile.cl}

\title{Quantitative mean-field limit for interacting  branching  diffusions}
\author{\authora \and \authorb}
\date{\today}

\begin{document}

\maketitle

\begin{abstract}
  We establish an explicit rate of convergence for some systems of mean-field interacting diffusions with logistic binary  branching towards the solutions of nonlinear evolution equations with non-local  self-diffusion and logistic mass growth, shown to describe their large population limits in \cite{FM2015}. The proof relies on a novel coupling argument for binary branching diffusions based on optimal transport, which allows us to sharply mimic the trajectory of the interacting binary branching population by certain system of independent particles with suitably distributed random space-time births. We are thus able to derive an optimal convergence rate, in the dual bounded-Lipschitz distance on finite measures, for the empirical measure of the population, from the convergence rate in  2-Wasserstein distance  of empirical distributions of i.i.d. samples. Our approach  and results extend techniques and ideas on propagation of chaos from kinetic models to stochastic systems of interacting branching populations, and appear to be new in this setting, even in the simple case of pure binary branching diffusions. 

\end{abstract}

\section{Introduction and main result}

Mathematical models of interacting and randomly evolving populations have been intensively studied the last decades through probabilistic and analytic approaches. Both points of view can integrate several biologically or ecologically meaningful features including individuals' displacements, reproduction and deaths,   competition for resources, selection,  and dispersive or attractive interactions.  While PDE and analysis methods can provide  aggregate  deterministic descriptions of the collective or macroscopic behavior of  such populations (see \cite{SKT1979,CHS2018,CCH2014,DEF2018,FF2013} and \cite{CDJ2018}, to name but a few works),  probabilistic methods have  successfully been employed  to describe the random behaviors  and interactions  of individuals at the microscopic (or finite population) level. Moreover, probabilistic weak convergence tools can be used to justify, in a rigorous way,  how certain  nonlinear evolution PDEs  arise as  limits  in law of the empirical processes of  individual-based models, when the population size  goes to infinity  (see for example \cite{FM2004,BM2015,FM2015,CDJ2019} and \cite{CDHJ2020}). Nevertheless, although it is clear that certain law of large numbers for exchangeable random systems underlies the passage from the microscopic to the macroscopic scale in branching population models, the speed of this convergence is not explicitly known, even in the simple case of pure binary  branching diffusions.  

In this work, we  develop a probabilistic  approach to obtain quantitative convergence estimates for the large population limit of a  general class of  spatially branching  diffusions with logistic growth and mean-field interactive spatial dynamics.
The population and its evolution are described by a right-continuous measure-valued Markov process taking values for fixed $K\in \NN\setminus \{0\}$ in the space of weighted finite point measures over $\mathbb{R}^d$
\[
  {\mathcal{M}}^K(\RR^d) \coloneqq \left\{ \frac{1}{K} \sum_{n=1}^N \delta_{x^n} : x^n\in \mathbb{R}^d,N\in \NN\setminus\{0\} \right\}\subseteq\mathcal{M}^+(\RR^d).
\]
Here, $\mathcal{M}^+(\RR^d)$ stands for the space of finite nonnegative measures on $\RR^d$ endowed with the weak topology and $\delta_x$ is the Dirac mass at $x\in \RR^d$. We denote said process by 
\begin{equation*}
  \mu^{K}_t = \frac{1}{K}\sum_{n=1}^{N^{K}_t} \delta_{X^{n,K}_t}, \quad t\geq 0, 
\end{equation*}
where $N^{K}_t \coloneqq K  \langle \mu_t^K, 1\rangle \in\mathbb{N}$ is the number of living individuals at time $t\geq 0$ and $X^{1,K}_t,\dots ,X^{\mbox{\scalebox{0.7}{$N^K_t$}}, K}_t$ are  their  positions in $\mathbb{R}^d$.
The parameter $K$ measures the population size and can be interpreted as  the carrying capacity of the underlying environment (see \cite{BM2015}).

The dynamics of  $(\mu^K_t)_{t\geq 0}$ is summarized as follows:
\begin{itemize}
  \item The initial population is described  by a random measure $\mu_0^K \in \mathcal{M}^K(\RR^d)$.
  \item Each  living individual  carries  at each instant  $t>0$ two clocks independent  between them: one reproduction clock, exponential of parameter $r>0$ and independent of  everything else in the system,  and one mortality clock, conditionally exponential of parameter $c N_t^K\hspace*{-0.4ex}/\hspace*{-0.1ex}K$, for $c>0$, given the population size $N_t^K$. If the reproduction clock of a particle rings at time $t$ when at position $x$, it gives birth to a new particle at that same position. If the mortality clock rings the particle disappears. Equivalently, the process jumps from  $\mu^K_{t-}$ to $ \mu^K_t = \mu^K_{t-} +K^{-1}\delta_x $ in the first case and to $ \mu^K_t=\mu^K_{t-} -K^{-1}\delta_x $  in the second. 
  \item Between birth or death events, for  each $n=1,..., N_t^K$ the   individual $X^{n,K}_t$ evolves according to the diffusion process
    \begin{align*}
      \dd X_t^{n,K}= b\bigl(X_t^{n,K},H*\mu_t^{K}(X_t^{n,K})\bigr)\,\dd t +  \sigma\bigl(X_t^{n,K},G*\mu_t^{K}(X_t^{n,K})\bigr)\,\dd B_t^{n}, 
    \end{align*}
    where  $(B^n)_{n\geq 1} $ are  Brownian motions in $\RR^d$, independent between them and independent of  $\mu_0^K $ and of the birth and death clocks.   
    \end{itemize} 

This model is a subclass of the non-local Lotka-Volterra cross-diffusion systems introduced in \cite{FM2015} as  a microscopic, individual-based  counterpart of the  celebrated Shigesada-Kawasaki-Tera\-moto  cross-diffusion  system  \cite{SKT1979}.  Here, we consider a simplified setting, consisting in one single species with self-interaction at the individuals' displacements level,  and  we assume that the demographic parameters determining  births and deaths are spatially homogeneous. In particular, the competitive pressure exerted on each individual is global and proportional to the population size, which  corresponds to a  constant competition kernel  in the general model of \cite{FM2015}.  

Following  \cite{FM2015} one can prove that, when $K$ goes to infinity, for each $T>0$ the empirical measure process $(\mu_t^K)_{t\in[0,T]}$ converges in law (in the Skorokhod space of finite measure-valued paths on $ [0,T]$) to a deterministic continuous measure-valued function $(\mu_t)_{t\in[0,T]}$, which  is the unique weak solution of a non-local self-diffusion equation (see \eqref{crossdiffeq} below). The following are assumptions that ensure this convergence and  which will be required to establish  our main result. 

\medskip

\noindent\textbf{Hypothesis (H)}:
\hypertarget{h}{}
 
\begin{enumerate}
  \item[\hypertarget{h1}{H.1.}] $( \langle \mu_0^K, 1\rangle)_{K\in \NN\setminus \{0\}}$ converges  in law   as $K\to \infty$ to some deterministic value  in $(0,\infty)$.   Moreover, for each  $K\in\NN\setminus\{0\}$, conditionally  on $ \langle \mu_0^K, 1\rangle$ the $N_0^K = K  \langle \mu_0^K, 1\rangle $ atoms of $\mu_0^K$ are   i.i.d. random variables with common probability law $\bar{\mu}_0$ not depending on $K$.   
  \item[\hypertarget{h2}{H.2.}] The functions  $\sigma\colon\RR^d\times\RR_+\to \RR^{d\otimes d} $ and $b\colon\RR^d\times\RR_+\to\RR^d$  are Lipschitz continuous. Moreover, there exists $C_\sigma>0$ such that for each $x\in\RR^d$ and $v\in\RR_+$,
    \[ |\sigma(x,v)|\leq C_\sigma(1+|v|).\]
  \item[\hypertarget{h3}{H.3.}] The functions $G,H\colon \RR^d\to \RR_+$ are bounded and Lipschitz continuous.
\end{enumerate}
Under assumption (\hyperlink{h}{H}),  $(\mu^K_t)_{t\geq 0}$ is a  Markov process which has finitely many jumps  in each finite time interval and whose law is uniquely determined. See \cite{FM2015} for details  and  \cite{Daw1993}  for general background on measure-valued Markov processes.  

Let $a \coloneqq \sigma \sigma^\mathrm{t} $ and, given $\mu \in \mathcal{M}^+(\RR^d)$, define an operator acting on  $C^2(\RR^d)$ functions  $\phi$ by
\begin{align*}
  L_{\mu}\phi (x)=\frac{1}{2}\mathrm{Tr}\left(a(x,G*\mu(x))\mathrm{Hess}(\phi)(x)\right)+b(x,H*\mu(x))\cdot\nabla \phi(x).
\end{align*}
As a particular case of \cite[Theorem 3.1]{FM2015}, we have the following statement.
\begin{thm}\label{thm:largepopulationlimit} 
  Assume \textnormal{(\hyperlink{h}{H})} and that \scalebox{1}{$\sup_{K\in \NN\setminus \{0\}}\EE(\langle\mu_0^{K},1\rangle^p)<+\infty$} for some $p\geq3$. Define $\mu_0\in \mathcal{M}^+(\RR^d)$ as the limit in law \scalebox{1}{$\mu_0 \coloneqq \lim_{K\to \infty }  \langle \mu_0^K, 1\rangle  \bar{\mu}_0$}. The sequence of processes $(\mu^{K})_{K\in \NN\setminus \{0\}}$ converges in law in \linebreak $D([0,T],\mathcal{M}^+(\RR^d))$ as $K\to\infty$ to the unique (deterministic) continuous finite measure-valued function $(\mu_t)_{t\in[0,T]}$ solution of 
 \begin{equation}\label{crossdiffeq}
    \left\langle\mu_t,f(t,\cdot)\right\rangle = \left\langle\mu_0,f(0,\cdot)\right\rangle + \int_0^t\bigl\langle\mu_s,\partial_sf(s,\cdot)+L_{\mu_s}f(s,\cdot)+(r-c\langle\mu_s,1\rangle)f(s,\cdot)\bigr\rangle\,\dd s, \quad \forall t\in[0,T],
 \end{equation}
  for every $f\in C_b^{1,2}([0,T]\times\mathbb{R}^d)$ such that $\sup_{(t,x)\in[0,T]\times\mathbb{R}^d}(1+\vert x\vert)\vert\nabla f(t,x)\vert<\infty$.
\end{thm}
Notice that the total mass  $n_t\coloneqq \langle \mu_t , 1\rangle$ of the measure $\mu_t$ has an autonomous, logistic evolution in $(0,\infty)$: $  \partial_t n_t  = (r-c  n_t) n_t, \, t\geq 0$.

Our main result is the quantification of the convergence to the large population limit in Theorem \ref{thm:largepopulationlimit}.  Recall that the weak topology on the space $\mathcal{M}^+(\mathbb{R}^d)$ can be metrized by means of the dual bounded-Lipschitz norm, which we denote by $\Vert\cdot \Vert_{ {{\BL}^*}}$ (see Section \ref{sec:preliminaires} for details). We have:

\begin{thm}\label{thm:rateofconvergence}
  Assume \textnormal{(\hyperlink{h}{H})}, that $\sup_{K\in \NN\setminus \{0\}}\EE(\langle\mu_0^{K},1\rangle^p)<\infty$ for some $p\geq  4$, and that $\int_{\mathbb{R}^d}\vert x\vert^q\,\mu_0(\dd x) <\infty$ for some  $q>2$. Then, for all $K\in\NN\setminus\{0\}$ and  $T>0$ one has 
  \begin{align*}
    \sup_{t\in[0,T]}\EE\bigl(\bigl\Vert\mu_t^K-\mu_t\bigr\Vert_{{\BL}^*}\bigr)\leq C_T 
    \begin{cases}
      \Bigl(I_4(K)+K^{-\frac{1}{4}}+K^{-\frac{(q-2)}{2q}}\Bigr), & \text{if }d<4\text{ and }q\neq4,\\
      \Bigl(I_4(K) +K^{-\frac{1}{4}}(\log(1+K))^{\frac{1}{2}} +K^{-\frac{(q-2)}{2q}}\Bigr), & \text{if }d=4\text{ and }q\neq 4, \\
      \Bigl(I_4(K)+K^{-\frac{1}{d}}+K^{-\frac{(q-2)}{2q}}\Bigr), & \hspace{-1.72em} \text{if }d>4\text{ and }q\neq \frac{d}{(d-2)},
    \end{cases}
  \end{align*}
 where $ I_4(K)=\EE\bigl(| \langle\mu^K_0,1\rangle-\langle\mu_0,1\rangle|^4\bigr)^{\frac{1}{4}}$ and $C_T>0$ depends on $T, p,q$ and the data of the model.
\end{thm}

The fact that  $ \langle\mu^K_0,1\rangle$ converges at least as fast as $K^{-1/4}$  in $L^4$ to $ \langle\mu_0,1\rangle$ can be granted for large families of random measures satisfying  (\hyperlink{h1}{H.1})  (see Lemma \ref{lem:weakconv_mu0}  in  Section \ref{sec:propchaos}  for details and also for possible relaxations of assumption (\hyperlink{h1}{H.1})). The convergence rate in Theorem \ref{thm:rateofconvergence} thus essentially depends  non-increasingly on the dimension $d$,  and on the amount of finite moments of  the measure $\mu_0$. For modeling purposes, the most relevant setting is $d=3$, in which case the rate  is equivalent  to $K^{-1/4}$ if  $q\in [4,+\infty)$, or to  the slower  rate $K^{-(q-2)/(2q)}$ if $q\in(2,4)$.  We notice also that  the same result  can be obtained in the case that each individual of the population additionally  carries an  independent, autonomous exponential killing clock of a fixed parameter (with the natural modification of the limiting PDE).

To prove Theorem \ref{thm:rateofconvergence} we will extend to the  branching populations setting some probabilistic  coupling techniques, based on optimal transport,  recently developed to quantify propagation of chaos in binary interacting particle systems from kinetic theory \cite{CF2016,CF2018}. See  \cite{Szn1991} and \cite{Mel1996}  for  general background on propagation of chaos theory. 
  
In the next section, we establish some preliminary results and  present the strategy of the proof of Theorem \ref{thm:rateofconvergence},  along with an outline of  the remainder of the paper.   We shall also discuss the ideas underlying our approach and discuss some consequences of our main result,  in the light  of propagation of chaos  theory.

\section{Preliminaries and strategy of the proof}\label{sec:preliminaires}

Denote by  $\text{BL}(\mathbb{R}^d)$ the space of Lipschitz continuous bounded functions in $\mathbb{R}^d$ with the norm
\begin{equation*}
  \Vert\varphi\Vert_\BL=\sup_{x\neq y}\frac{\vert\varphi(x)-\varphi(y)\vert}{x-y}+\sup_x\vert\varphi(x)\vert, 
\end{equation*}
and by  $\Vert\cdot \Vert_{ {{\BL}^*}}$ the corresponding dual norm on the space  $\mathcal{M}(\mathbb{R}^d)$  of  finite signed measures on  $\mathbb{R}^d$. The induced distance
 \begin{equation*}
  \Vert\mu-\nu\Vert_ {{\BL}^*}=\sup_{\Vert\varphi\Vert_\BL\leq1}\vert\langle\mu-\nu,\varphi\rangle\vert, 
\end{equation*}
is well known to generate the weak convergence topology on $\mathcal{M}^+(\mathbb{R}^d)$. The subspace of $\mathcal{M}^+(\mathbb{R}^d)$ of probability measures is denoted by $\mathcal{P}(\mathbb{R}^d)$.
Given a measure $\mu \in \mathcal{M}^+(\mathbb{R}^d)$,  its $q$-th moment for $q\in [1,\infty)$  is denoted by
\[M_q(\mu)=\int_{\mathbb{R}^d}\vert x\vert^q\,\mu(\dd x).\]

For $p\in [1,\infty)$,  the  $p$-Wasserstein distance $W_p(\mu,\nu)$ between two probability measures $\mu, \nu \in \mathcal{P}(\mathbb{R}^d)$  is defined by
\[W_p(\mu,\nu)=\biggl(\inf_{\pi\in\Pi(\mu,\nu)}\int_{\mathbb{R}^d\times \mathbb{R}^d}\left\vert x-y\right\vert^p\,\pi(\dd x,\dd y)\biggr)^\frac{1}{p},\]
where $\Pi(\mu,\nu)$ is the set of probability measures over $\RR^d\times\RR^d$ that have $\mu$ and $\nu$ as first and second marginals respectively. A coupling $\pi \in \Pi(\mu,\nu)$ realizing the infimum always exists and is called an \textit{optimal coupling}  between $\mu$ and $\nu$  for the transport cost  $c(x,y)=  |x-y|^p$. $W_p$ defines a  complete distance if  restricted to the space $\{\mu\in \mathcal{P}(\mathbb{R}^d): M_p(\mu)<\infty\}$  and is equivalent therein to the weak topology strengthened with the convergence of  $p$-th moments. See \cite{Vil2009} for background.
 
For every  $\mu\in\mathcal{M}^+(\mathbb{R}^d)$,  we  will throughout denote by  $\bar{\mu}$ the probability measure  on $\RR^d$ obtained from it by normalization:
\[ \bar{\mu} \coloneqq \frac{1}{\langle \mu,1\rangle} \mu \in {\mathcal{P}}(\RR^d). \]

The following simple relations  for finite measures, proved in \hyperlink{appendix}{Appendix},  will be useful. 
\begin{lem}\label{lem:normwasserstein}
  Let $\mu,\nu\in\mathcal{M}^+(\mathbb{R}^d)$. We have
  \[ \Vert\mu-\nu\Vert_ {{\BL}^*}\leq \langle\mu,1\rangle \Vert\bar{\mu}-\bar{\nu}\Vert_ {{\BL}^*}+\big\vert \langle\mu,1\rangle -\langle\nu,1\rangle \big\vert, \]
  and
  \[ \Vert\bar{\mu}-\bar{\nu}\Vert_ {{\BL}^*}\leq \inf_{\pi\in\Pi(\bar{\mu},\bar{\nu})}\int \vert x-y\vert\wedge2\,\pi(\dd x,\dd y)\leq W_1(\bar{\mu},\bar{\nu}). \]
\end{lem}

The basic estimate on which our main result relies, is the quantitative bound in $2$-Wasserstein distance for empirical measures of i.i.d. samples, proved in  \cite{FG2015} and stated next for convenience. 

\begin{thm}\label{thm:FournierGuillin}
  Let $\bar{\mu} \in \mathcal{P}(\RR^d)$ and $(X^n)_{n\in\NN}$ be i.i.d. random variables with law $\bar{\mu}$. Assume  $M_q(\bar{\mu}) < \infty$ for some $q > 2$. Then, there exists a constant $C_{d,q}>0$ depending only on $d$ and $q$ such that, for all $N\in\NN\setminus\{0\}$,
  \begin{equation*}
    \EE\biggl( W_2^2\biggl(\frac{1}{N}\sum_{n=1}^N \delta_{X^n},\bar{\mu}\biggr) \biggr) \leq C_{d,q}M_q^\frac{2}{q}(\bar{\mu})\, R_{d,q}(N), 
  \end{equation*}
  where  $R_{d,q}\colon\NN \setminus \{0\}\to \RR_+$ is defined by  
  \begin{equation*}
    R_{d,q}(N) \coloneqq \begin{cases}
      N^{-\frac{1}{2}} + N^{-\frac{(q-2)}{q}}, & \text{if } d<4 \text{ and } q\neq 4, \\
      N^{-\frac{1}{2}}\log(1+N) + N^{-\frac{(q-2)}{q}}, & \text{if } d=4 \text{ and } q\neq 4, \\
      N^{-\frac{2}{d}} + N^{-\frac{(q-2)}{q}}, & \text{if } d>4 \text{ and } q\neq \frac{d}{d-2}.
    \end{cases}
  \end{equation*}
\end{thm}

One can deduce analogous estimates for random empirical measures in $ \mathcal{M}^K(\RR^d)$ whose atoms satisfy a certain conditionally independence property. See \hyperlink{appendix}{Appendix} for the proof of the next result.

\begin{lem}\label{lem:boundNW}
  Let  $\mu \in \mathcal{M}^+(\RR^d)$ be such that  $M_q(\mu)<\infty$ for some $q>2$ and  let $(N,\nu^K)$ be a random variable in $\NN\times \mathcal{M}^K(\RR^d)$  such that   $\EE(N)<\infty$ and, conditionally on $N$,   $\nu^K$   is supported on $N$ atoms that are i.i.d. random variables of law $\bar{\mu} $.  
  Then, there exists a constant $C_{d,q}>0$ that depends only on $d,q$ such that
  \begin{equation*}
    \EE\Bigl(\frac{N}{K}W_2^2\bigl(\bar{\nu}^K, \bar{\mu} \bigr)\Bigr)\leq C_{d,q} M_q^\frac{2}{q}(\bar{\mu})  \,  \EE(1\vee (N/K )) R_{d,q}(K). 
  \end{equation*} 
\end{lem}

Notice that under assumption (\hyperlink{h1}{H.1}),  Lemma \ref{lem:boundNW} immediately provides quantitative estimates for \linebreak $W_2^2(\bar{\mu}_t^K,\bar{\mu}_t)  $ when $t=0$; however, the required conditional independence property is lost as soon as $t>0$,  even in the case of pure branching diffusions.

\subsection{Proof strategy and plan of the paper}\label{sec:strategy}

The proof of Theorem \ref{thm:rateofconvergence} is based  on the construction, for each $K$, of  a coupling between the system  $(\mu_t^K)_{t\geq 0}$, and an auxiliary system  of particles in $ \mathcal{M}^K(\RR^d)$ denoted by
\[ \nu^{K}_t \coloneqq \frac{1}{K}\sum_{n=1}^{N^{K}_t} \delta_{Y^{n,K}_t}, \quad t\geq 0,\]
such that the following condition holds:

\medskip 

\noindent\textbf{Condition (C)}: 
\hypertarget{c}{}

\begin{enumerate}
  \item[\hypertarget{c1}{C.1.}]  $ \nu_0^K = \mu_0^K  $ and  $K \langle \nu_t^K, 1\rangle =K  \langle \mu_t^K, 1\rangle = N_ t^K$  for all $t\geq 0$ almost surely.
  \item[\hypertarget{c2}{C.2.}] For each $t\geq 0$, conditionally on   $ \langle \nu_t^K, 1\rangle $, the atoms of $ \nu_t^K$   are i.i.d. random variables of law $\bar{\mu}_ t$.
  \item[\hypertarget{c3}{C.3.}] For each $T>0$ there is a constant $C_T>0$ depending on  $T$ and on the data of  Theorem \ref{thm:rateofconvergence} such that
  \[
    \EE\Bigl(\frac{N_t^K}{K}W_2^2\bigl(\bar{\nu}_t^K, \bar{\mu}_t^K \bigr)\Bigr)\leq C_T  \,  (R_{d,q}(K) +I_4^2(K)) .
  \]
\end{enumerate}

Let us describe how this construction is used and  how the arguments of the proof will unfold in the remainder of the paper: 
\begin{itemize}
  \item Thanks to  condition (\hyperlink{c1}{C.1}), Lemma \ref{lem:normwasserstein} and some auxiliary estimates,  obtaining the searched  bound boils down, by triangular inequality, to controlling on finite time intervals the quantities  \\ $\EE\bigl(\frac{N_t^K}{K}W_2^2\bigl(\bar{\nu}_t^K, \bar{\mu}_t^K \bigr)\bigr) $ and $ \EE\bigl(\frac{N_t^K}{K}W_2^2\bigl(\bar{\nu}_t^K, \bar{\mu}_t \bigr)\bigr) $.
  \item  Condition (\hyperlink{c2}{C.2})  and  Lemma  \ref{lem:boundNW}  together imply that the quantity  $ \EE\bigl(\frac{N_t^K}{K}W_2^2\bigl(\bar{\nu}_t^K, \bar{\mu}_t \bigr)\bigr) $   is bounded by $C_T  \,  R_{d,q}(K)$.
  \item  The previous facts and the bound in condition (\hyperlink{c3}{C.3}) together  will imply the bound asserted in Theorem \ref{thm:rateofconvergence}.  
\end{itemize} 

In Section \ref{sec:pathwiseconstruction} we explicitly construct the coupled particle systems, $(\mu_t^K)_{t \geq 0}$ and $(\nu_t^K)_{t \geq 0}$, in terms of common Brownian motions and a suitable Poisson point measure. 
In this construction, condition (\hyperlink{c1}{C.1}) is simply verified  since  the birth and death events of the two systems will be simultaneous, and they both will start from the same state.
In order to ensure condition (\hyperlink{c2}{C.2}), each atom  $Y_t^{n,K}$  of $\nu_t^K$  will be defined as a suitable McKean-Vlasov diffusion (defined in Proposition \ref{prop:nonlinearprocess}), whose law at each time $t $ from its birth-time on is given by $\bar{\mu}_t$, and it will  evolve   independently  of  everything else in the system.
 
The crucial, far from trivial feature of the coupling is condition (\hyperlink{c3}{C.3}). Since
\[
  \EE\Bigl(\frac{N_t^K}{K}W_2^2\bigl(\bar{\nu}_t^K, \bar{\mu}_t^K \bigr)\Bigr)\leq  \EE\biggl(\frac{N_t^K}{K} \frac{1}{N_t^K} \sum_{n=1}^{N_t^K} \| X_t^{n,K}- Y_t^{n,K}\|^2 \biggr) = \EE\biggl(\frac{1}{K} \sum_{n=1}^{N_t^K} \| X_t^{n,K}- Y_t^{n,K}\|^2 \biggr), 
\]
by using the same Brownian motion to drive the  two atoms $(X_t^{n,K},Y_t^{n,K})$  and  relying on the  Lipschitz character of the coefficients,  we will be able to ensure condition (\hyperlink{c3}{C.3}) by coupling the birth positions of the two paired particles in the best possible way, in the $L^2-$distance sense. This is where optimal transport  ideas and techniques introduced  in  \cite{CF2016,CF2018} will come into play.   
Indeed,  on one hand, the birth position of a  new particle in the system $(\nu_t^K)_{t\geq 0 } $,  born at a random time $s$,  will be sampled  in $\RR^d$  according to the law  $\bar{\mu}_s$. On the other, choosing randomly a particle that branches  at time $s$
in  system  $(\mu_t^K)_{t\geq 0} $ is equivalent to sampling a position  in $\RR^d$  at that time,  according to the empirical law  $\bar{\mu}^K_{s-}$.   Thus, the optimal way to couple a pair of atoms  in the two systems at their birth  time $s$ is to sample them simultaneously  from the optimal coupling for $W^2_2$  of the law $\bar{\mu}_s$ and  the (random) law $\bar{\mu}^K_{s-}$. This joint sampling must be done in a measurable way in terms of the state of the process at time $s-$, which requires  using a non-trivial construction  from   \cite{CF2016}, adapted to our setting in Lemma \ref{lem:optimal_coupling}.

In Section \ref{sec:bd} we consider the simpler case of pure binary branching processes  (i.e. with no mean-field interaction between the particles  nor competition). We establish some auxiliary estimates, we prove that condition  (\hyperlink{c3}{C.3}) holds in that specific case, and we deduce Theorem \ref{thm:rateofconvergence}  with slightly better bounds. 
 
In Section \ref{sec:bdi} we follow similar steps to deduce the proof of Theorem \ref{thm:rateofconvergence}  as stated in the general case. 

Finally, in the last section we discuss  potential  extensions of the developed ideas and  results to more general branching population models.

Before delving into the proofs, we briefly discuss the relation of our results with the propagation of chaos property in mean-field interacting particle systems, and we make some remarks concerning assumption (\hyperlink{h1}{H.1}) and related conditions  in that framework. 

\subsection{Propagation of chaos for interacting branching diffusions}\label{sec:propchaos}

It is well known that  convergence of the empirical probability distribution  of   $N$ exchangeable particles to some deterministic probability measure,  when $N$ is a non-random integer that goes to infinity, is  equivalent to the property of propagation of chaos,  or asymptotic independence of the particles \cite{Szn1991,Mel1996}.   We next introduce an extended notion of it,  whereby Theorem  \ref{thm:rateofconvergence} can be viewed  as a propagation of chaos result.  

\begin{dfn}
Let $(N^K)_{K\in\NN \setminus \{0\}}$  be random variables in $\NN$ going in law  to  $\infty$  as $K\to \infty$. We say a family $((Y^{1,K}, \dots, Y^{N^K,K}))_{K\in\NN \setminus \{0\}} $ of random vectors, $(\RR^d)^{N^K}\hspace{-0.2em}$-valued and exchangeable conditionally on $N^K$ for each $K$, is conditionally $P$-\textit{chaotic} given  $(N^K)_{K\in\NN \setminus \{0\}}$ if for some $P\in {\mathcal{P}}(\RR^d)$ and every  $j \in \NN\setminus\{0\}$ the (random) conditional laws  $\bigl( {\cal L}(Y^{1,K}, \dots, Y^{j\wedge N^K,K} \mid N^K) \bigr)_{K\in\NN \setminus \{0\}}$ given $N^K$ and the event $\{N^K\geq j\}$ converge in distribution in $\mathcal{P}((\RR^d)^j)$ to $P^{\otimes j}$  as $K\to \infty$.
\end{dfn}

In the case that $N^K= K $  is deterministic for all $K\in\NN \setminus \{0\}$,  one recovers the well known notion of $P$-chaoticity  \cite{Szn1991,Mel1996}.  Under the same assumptions of Theorem \ref{thm:rateofconvergence} we deduce the following result, proved in Section \ref{sec:bdi}.

\begin{cor}\label{cor:condiquantipropchaos}
  For each $t\geq 0$ the family $((X_t^{1,K},\dots, X_t^{\mbox{\scalebox{0.7}{$N^K_t,K$}}}))_{K\in \NN\setminus \{0\}}$ is conditionally $P$-chaotic  given $(N^K_t)_{K\in \NN\setminus \{0\}}$  with $P=\mu_t / \langle  \mu_t , 1\rangle$.
\end{cor}
 
We end this section gathering some remarks on assumption (\hyperlink{h1}{H.1}), including its  possible relaxation to a chaoticity condition. The proof of this result is given in the \hyperlink{appendix}{Appendix}.

\begin{lem}\phantomsection\label{lem:weakconv_mu0}
  \begin{itemize}
    \item[\textnormal{a)}] Under \textnormal{(\hyperlink{h1}{H.1})},    $(\mu_0^{K})_{K\in \NN\setminus \{0\}}$ converges in law to the deterministic finite measure 
  $$
      \mu_0 \coloneqq \lim_{K\to \infty }  \langle \mu_0^K, 1\rangle  \bar{\mu}_0.
   $$
    \item[\textnormal{b)}] The same conclusion as in  \textnormal{a)} holds  if $( \langle \mu_0^K, 1\rangle)_{K\in \NN\setminus \{0\}}$ converges  in law  as $K\to \infty$ to a constant in $(0,\infty)$ and there exists a  $\bar{\mu}_0$-chaotic   family  of exchangeable random vectors  $( (Y^{1,N}, \dots, Y^{N,N}) : N\in \NN\setminus\{0\}) $   such that for all $K$,  conditionally  on $ K \langle \mu_0^K, 1\rangle = N $, the set of atoms of $\mu_0^K$  has the  same  law as $(Y^{1,N}, \dots, Y^{N,N})$.
\item[\textnormal{c)}] \textnormal{(\hyperlink{h1}{H.1})} holds  if  $K \mu_0^K$ is  for each $K$  a Poisson point measure on $\RR^d$  of intensity $K \nu_0$   with   $\nu_0 \in \mathcal{M}^+ (\RR^d)$ fixed. In this case,   $\mu_0$ defined in \textnormal{a)} is equal to $\nu_0$ and, moreover, we have  $I_4(K)  \leq C K^{-1/2}.$
  \end{itemize}
\end{lem}

\begin{rmk} 
If instead of  \textnormal{(\hyperlink{h1}{H.1})} we assume that the initial data $\mu^K_0$ satisfies  only the condition  in Lemma \ref{lem:weakconv_mu0} \textnormal{b)}, Theorem  \ref{thm:rateofconvergence} still holds  but with an additional term on the r.h.s. of generic form: $C_T    \EE\bigl(\frac{1}{K}\sum_{n=1}^{\mbox{\scalebox{0.7}{$N_0^{K}$}}}\left\| X_0^n-Y_0^n\right\|^2\bigr)$ where, conditionally  on $\{N_0^K=N\}$,   $\bigl( (X_0^1,\dots,X_0^N) , (Y_0^1,\dots,Y_0^N) \bigr)$ is  for each $N,K\in \NN$ a coupling of  the $N$ atoms of $\mu_0^K$ and an i.i.d. sample of size $N$ of the law $\bar{\mu}_0$.  See Remark \ref{rmk:initialcondcoupling} for details and for the optimal  value of this term. 
\end{rmk}

\section{Pathwise constructions and coupling algorithm}\label{sec:pathwiseconstruction}

For the rest of the article we will omit the superscripts $K$ in the particles' positions, e.g. we write  $\bigl(X^{1}_t,\dots ,X^{\mbox{\scalebox{0.7}{$N^K_t$}}}_t\bigr)=\bigl(X^{1,K}_t,\dots , X^{\mbox{\scalebox{0.7}{$N^K_t$}},K}_t\bigr)$ since we will be  working with fixed $K\in \NN\setminus \{0\}$ and  no ambiguity is possible. 

We will construct  both  systems $(\mu_t^K= \frac{1}{K}\sum_{n=1}^{N^{K}_t} \delta_{X^{n}_t})_{t\geq 0} $  and $(\nu_t^K= \frac{1}{K}\sum_{n=1}^{N^{K}_t} \delta_{Y^{n}_t})_{t\geq 0} $  from the following set of independent stochastic inputs defined in a common complete probability space $(\Omega, {\mathcal{F}},\PP)$:
\begin{itemize}
  \item  A sequence $(W^j)_{j\geq 1}$ of independent Brownian motions in $\RR^d$.
  \item A Poisson point measure $\mathcal{N}(\dd s,\dd \rho,d\theta)$  on  $[0,\infty)\times[0,\infty)\times[0,\infty)$, with intensity $\dd s \otimes \dd \rho \otimes \dd \theta$.
  \item A sequence $(Z_0^j)_{j\geq 1}$ of i.i.d. random vectors of law $\bar{\mu}_0$.
  \item A random variable  $N_0^K$ in $\NN$.
\end{itemize}

We will also make use of a special diffusion process considered in \cite{FM2015}, which can be seen as a nonlinear process in the sense of McKean \cite{Szn1991,Mel1996}.
In the current setting, this process is characterized next.

\begin{prop}\label{prop:nonlinearprocess}
  Let $(\mu_t)_{t\geq 0}$ be the unique weak solution  in ${\mathcal{M}}^+(\RR^d)$ of the nonlinear equation
  \begin{equation}\label{bdi:limiteq}
    \frac{\partial\mu_t}{\partial t}=L^*_{\mu_t}\mu_t+\bigl(r-c\langle\mu_t,1\rangle\bigr)\mu_t,
  \end{equation}
  given by Theorem \ref{thm:largepopulationlimit},  with initial condition $\mu_0$. Let $W$ be a  $d-$dimensional Brownian motion and $Y_0$ an independent random variable in  $\RR^d$ with law $\bar{\mu}_0$.  There is pathwise existence and uniqueness for the SDE
  \begin{equation}\label{eq:nonlindiffusion}
    Y_t=Y_0 +\int_0^tb(Y_s,H*\mu_s(Y_s)) \,\dd s+\int_0^t\sigma(Y_s,G*\mu_s(Y_s)) \,\dd W_s.
  \end{equation}
  Moreover, the  flow of  time-marginal laws of $(Y_t)_{t\geq 0}$   is the unique weak solution  $(\bar{\mu}_t)_{t\geq 0}$  in ${\mathcal{P}}(\RR^d)$ of the (linear, non-homogeneous in time)  Fokker-Planck equation 
  \begin{equation}\label{bdi:normalizedeq}
    \frac{\partial\bar{\mu}_t}{\partial t}=L_{\mu_t}^*\bar{\mu}_t,
  \end{equation}
  with respect to test functions as in Theorem \ref{thm:largepopulationlimit}, and  we have $\bar{\mu}_t=\mu_t /\langle \mu_t ,1\rangle$ for all $t\geq 0$. Last, for every bounded measurable function $f\colon\RR^d\to \RR$  we have $\langle \mu_t ,f\rangle =  \EE( f(Y_t) n_t),$ where $n_t$ is the unique solution with  $n_0= \langle \mu_0 ,1\rangle$ of the logistic equation 
  \begin{equation}\label{eq:logisticODE}
    \dd n_t =   \bigl(r-cn_t\bigr) n_t \,\dd t. 
  \end{equation}
\end{prop}

The proof of Proposition \ref{prop:nonlinearprocess} is postponed to Section  \ref{sec:bdi}. 

\begin{rmk}\phantomsection\label{rmk:nonlinearprocess}
  \begin{itemize}
    \item[\textnormal{a)}]  The pathwise properties of the SDE \eqref{eq:nonlindiffusion} stated in  Proposition \ref{prop:nonlinearprocess} imply for  fixed ${\tau}>0$  that if $Y_{\tau}'$ is a random variable of law $\bar{\mu}_{\tau}$  ,  independent of $W$,  then the solution $(Y_t')_{t\geq {\tau}}$  of the SDE
    \begin{equation*}
      Y_t'=Y_{\tau}' + \int_{\tau}^tb(Y_s',H*\mu_s(Y_s')) \,\dd s+\int_{\tau}^t\sigma(Y_s',G*\mu_s(Y_s')) \,\dd W_s,
    \end{equation*}
    has the same law as $(Y_t)_{t\geq {\tau}}$.  In particular, $Y_t'$ has law $\bar{\mu}_t$ for all $t\geq {\tau}$. 
  \item[\textnormal{b)}]   When $\sigma $ and $\, b$ depend only on the position and not on $\mu$, the process \eqref{eq:nonlindiffusion} is the  standard diffusion associated with the generator
    \begin{equation}\label{bd:diffusiongen}
      Lf(x)=\frac{1}{2}\mathrm{Tr}\big(a(x)\mathrm{Hess}f(x))+b(x)\cdot\nabla f(x), 
    \end{equation}
    which  in that case also drives each of the particles of the branching system $(\mu^K_t)_{t\geq 0}$.  Notice also that in this setting, thanks to the Lipschitz character of the coefficients, if $\bar{\mu}_0$ has finite moments of order $q\geq2$, then  finiteness of these moments is  uniformly  propagated over any time  interval $[0,T]$. 
  \end{itemize} 
\end{rmk}

Last, the following construction, based on  optimal transport and adapted from \cite{CF2016}, will allow us to couple the births positions in the two systems in the most efficient way, as discussed in Section  \ref{sec:strategy}.

\begin{lem}\label{lem:optimal_coupling}
  Let  $\mathbf{i}\colon\mathbb{R}\to\mathbb{N}$  denote the function defined by 
  \[\rho \mapsto \mathbf{i}(\rho)=\lfloor \rho\rfloor + 1, \]
  and let $N$ be a positive integer.  Let also $(\bar{\mu}_t)_{t\geq0}$ be a flow of probability measures with finite second order moments that is  weakly continuous. There exists a measurable mapping
  \[ \Lambda^N\colon\mathbb{R}_+\times(\mathbb{R}^d)^N\times[0,N)\to\mathbb{R}^d,\quad (t,\mathbf{x},\rho)\mapsto\Lambda^N_t(\mathbf{x},\rho), \]
  with the following properties:
  \begin{itemize}
  \item For every $t\geq0$ and $\mathbf{x}=(x^1,\dots, x^N)\in(\mathbb{R}^d)^N$, if $\rho$ is uniformly chosen from $[0,N)$, then the pair $(\Lambda^N_t(\mathbf{x},\rho),x^{\mathbf{i}(\rho)})$ is an optimal coupling between $\bar{\mu}_t$ and $\frac{1}{N}\sum_{i=1}^N\delta_{x^i}$ with respect to the cost function $(u,v)\mapsto \vert u-v\vert^2$.
  \item If  $\, \mathbf{Y}$ is any exchangeable random vector in $(\mathbb{R}^d)^N$, then $\mathbb{E}\big(\int_{j-1}^j\phi(\Lambda^N_t(\mathbf{Y},\tau))\dd \tau\big)=\left\langle\bar{\mu}_t,\phi\right\rangle$ for any $j\in\lbrace1,\dots,N\rbrace$, and any bounded measurable function $\phi$.
  \item The function $\Lambda \colon \NN  \times \RR_+ \times \bigl(  \bigcup_{N\in \NN\setminus\{0\}} ( \RR^d )^N \bigr) \times \RR_+ \to \RR^d $ given by  
  \[ \Lambda(N,t, \mathbf{x},\rho)= \Lambda^N_t\bigl((x^n)_{n=1}^N  ,\rho\wedge N\bigr), \]
  if $\mathbf{x}= (x^n)_{n=1}^N \in (\RR^d)^N$, and $0\in  \RR^d $ otherwise,  is measurable. 
  \end{itemize}
\end{lem}

\begin{proof}
  Everything is proved in \cite[Lemma 3]{CF2016} except for the last assertion, which follows noting that $\Lambda^{-1}(A) = \bigcup_{N\neq 0}  \{ N\} \times (\Lambda^N)^{-1}(A )  $ is a measurable set for any Borel set $A\in \RR^d$ such that $0\not\in A$, and $\Lambda^{-1}(\{0\})= \left(  \bigcup_{N\neq 0 }  \{ N\} \times \RR_+\times  \cup_{n \neq N} (\RR^d)^n \times \RR_+ \right) \cup \left(  \bigcup_{N\neq 0}  \{ N\} \times (\Lambda^N)^{-1}(\{0\}) \right).$
\end{proof}

\subsection{Coupling algorithm}

Before giving the algorithm, we also  introduce a sequence of \textit{labelling processes}
\[
  (j_t(n) : t\geq 0 )_{n\geq 1},
\]
taking values in the positive integers, that will be  dynamically defined to select from $(W^j)_{j\geq 1}$ the Brownian motions driving each  coupled pairs  of particles $(X_t^n,Y_t^n)$,  in between  reproduction or death events. 

The  systems   $(\mu_t^K= \frac{1}{K}\sum_{n=1}^{N^{K}_t} \delta_{X^{n}_t})_{t\geq 0} $   and  $(\nu_t^K= \frac{1}{K}\sum_{n=1}^{N^{K}_t} \delta_{Y^{n}_t})_{t\geq 0} $   are then constructed simultaneously, through the following algorithm.
\medskip

\noindent\textbf{Algorithm (A)}:  
\hypertarget{a}{}

\begin{itemize}
  \item[0.] We set  $Y_0^n= X_0^n=Z_0^n $ for $n\in \{1,\dots, N^K _0 \} $ and $ \mu^{K}_0 =  \nu^{K}_0 = \frac{1}{K}\sum_{n=1}^{N^K _0}  \delta_{Z^{n}_0}.$ We  also set two counters:    $\overline{N}_0^K=N^K_0$ and $m=0$, and we define $T_0=0$.  Last, we initialize 
  $j_0(n)=n$ for all $n\geq 1$.
  \item[1.] For $t\geq T_m$, we set $j_t(n)= j_{T_m}(n)$ and $\dd B_t^{n}= \dd W_t^{j_t(n)}$, $n\geq 1$, and we define the  dynamics of the two populations  by:
    \begin{equation*}
      X_t^{n}= X_{T_m} + \int_{T_m}^t b\big(X_s^{n},H*\mu_s^{K}(X_s^{n})\big)\,\dd s +  \int_{T_m}^t\sigma\big(X_s^{n},G*\mu_s^{K}(X_s^{n})\big)\dd B_s^{n},   \quad  n=1,\dots, N_{T_m}^K,
    \end{equation*}
    and
    \begin{equation*} 
      Y_t^{n} = Y_{T_m} + \int_{T_m}^t b\big(Y_s^{n},H*\mu_s(Y_s^{n})\big)\,\dd s +  \int_{T_m}^t\sigma\big(Y_s^{n},G*\mu_s(Y_s^{n})\big)\dd B_s^{n},  \quad  n=1,\dots, N_{T_m}^K,
    \end{equation*}
    until the first time $t>T_m$ with $ (t,\rho,\theta)$   an atom of ${\mathcal{N}}$,  such that
    \[ \rho \leq N_{T_m}^K \quad\text{and}\quad \theta \leq r+ c  \frac{N_{T_m}^K}{K} . \]
    We then set $T_{m+1}=t$.
  \item[2.] For $ (t,\rho,\theta)=(T_{m+1},\rho,\theta)$  as before,
    \begin{itemize}
      \item If $\theta \leq r$, we update $N_t^K \coloneqq N_{t-}^K+1$ and $\overline{N}_t^K \coloneqq \overline{N}_{t-}^K+ 1$, then we define:
        \[ X_t^{N_t^K} \coloneqq X_{t-}^{\mathbf{i}(\rho)} \quad \text{and}\quad Y_t^{N_t^K} \coloneqq  \Lambda^{N^K_{t-}}_t\Bigl((X_{t-}^n)_{n =1}^{N^K_{t-}},\rho\Bigr). \]
      \item If $r < \theta \leq  r + c  N_{T_m}^K /K$, we update $N_t^K \coloneqq N_{t-}^K-1$,  then we redefine:
        \begin{align*}
        \bigl(X_t^{\mathbf{i}(\rho)}, X_t^{\mathbf{i}(\rho)+1} ,\dots,  X_t^{N_t^K} \bigr) &\coloneqq \bigl(X_{t-}^{\mathbf{i}(\rho)+1}, X_{t-}^{\mathbf{i}(\rho)+2} ,\dots,  X_{t-}^{N_{t-}^K} \bigr), \\
      \bigl(Y_t^{\mathbf{i}(\rho)}, Y_t^{\mathbf{i}(\rho)+1} ,\dots,  Y_t^{N_t^K} \bigr) &\coloneqq \bigl(Y_{t-}^{\mathbf{i}(\rho)+1}, Y_{t-}^{\mathbf{i}(\rho)+2} ,\dots,  Y_{t-}^{N_{t-}^K} \bigr),
      \end{align*}
        and we set $j_t(n) \coloneqq j_{t-}(n+1)$ for all $n \geq \mathbf{i}(\rho)$.
    \end{itemize}
  \item[3.] We increase $m$ by one and go to Step 1.
\end{itemize}

Let us explain in words how  the algorithm works. The systems  $(\mu_t^K)_{t\geq 0} $  and $(\nu_t^K)_{t\geq 0} $   start at time $t=0$ from the same empirical measure, and pairs of particles are given birth  or die  in the two systems simultaneously  from then on. The variable  $N_t^K$   counts the current number of living particles in each system at time $t$.  The variable 
$\overline{N}_t^K$  in turn counts how many particles  have been alive in each of the two systems or, equivalently, how many  Brownian motions from   $(W^j)_{j\geq 1}$ have been used, during  the whole time interval $[0,t]$.  The usefulness of this counter will come clear shortly.

Now, given an atom  $(t,\rho, \theta)$, its coordinate $t$ is used to sample a proposal of a birth or dead time, and $\theta$ an ``action'' among those two, according to whether $\theta \leq r$ or $r < \theta \leq r + c N_{t-}^K/K$ respectively. 

In a birth event, $\rho \leq N_{t-}^K$ samples  two positions in space,  one distributed according to $\bar{\mu}_{t-}^K$ for the system $\mu^K$ and one according to $\bar{\mu}_t$ for the system $\nu^K$, which are  optimally coupled as explained before.  The pair of newborn particles picks upon birth at time $t$  a new, common driving Brownian motion  $(W_s^{\overline{N}^K})_{s\geq t}$ that is independent of the past of the systems.   

In a death event, $\rho \leq N_{t-}^K$ samples a uniformly distributed  atom from  $\bar{\mu}_{t-}^K$ for the system $\mu^K$ and  from $\bar{\nu}_{t-}^K$  for the system $\nu^K$, with equal index $\mathbf{i}(\rho)$.  The two corresponding particles are then removed,  and their common driving Brownian motion, which corresponds to some $W^j$ with $j\leq \overline{N}_t^K$, is discarded forever.  The indexes of the particles in the two systems are then updated, as well as the Brownian motions  from $(W^j)_{j\geq 1}$  labelled  $B^{\mathbf{i}(\rho)}, B^{\mathbf{i}(\rho)+1},... $,   in order that the particles still alive remain indexed by a full discrete interval  of the form $\{ 1,\dots,  N_t^K\}$, and that the underlying  Brownian motion $W^j$ driving each pair is preserved.    Notice that, due to this updating rule, for all times $t\geq 0$ we have  $j_t(N_t^K)= \overline{N}_t^K$. 

The system $(\nu_t^K)_{t\geq 0}$  satisfies  condition (\hyperlink{c1}{C.1}) by construction.  In the next paragraph, we will check  that it also satisfies condition (\hyperlink{c2}{C.2}). 
  
\subsection{Verification of condition \texorpdfstring{({\protect\hyperlink{c2}{C.2}})}{(C.2)}} 

We will denote by $( {\mathcal{F}}_t)_{t\geq 0}$ the complete filtration generated by all the random objects effectively employed in the algorithm until  each time:
\[ {\mathcal{F}}_t \coloneqq \overline{ \sigma \left(  N_0^K, (Z_0^n )_{n\in \{1,\dots, N^K _0 \} }, ( {\mathcal{N}} ((0,s], \cdot\,, \cdot\,) : s\leq t) ,  (B^n _s : s\leq t)_{n\in \{1,\dots, N^K _t \}}  \right)  }, \]
and by  $( {\mathcal{G}}_t)_{t\geq 0}$ its subfiltration 
\[ {\mathcal{G}}_t \coloneqq \overline{ \sigma \left(  N_s^K : s\leq t  \right)  }. \]

Notice that ${\mathcal{N}}$ is an  $({\mathcal{F}}_t)_{t\geq 0}$-Poisson process, and that the processes $(N_t^K)_{t\geq 0}, (\overline{N}_t^K)_{t\geq 0}$ and $(j_t(n) : t\geq 0 )$, $n\geq 1$ are adapted to  $( {\mathcal{G}}_t )_{t\geq 0}$. 

\begin{rmk}\label{rmk:atoms}
  Thanks to Lemma \ref{lem:optimal_coupling},  the mapping
  \[
    (t,\omega,\rho)\mapsto\left( \Lambda^{N^K_{t-}}_t\Bigl((X_{t-}^n)_{n =1}^{N^K_{t-}},\rho\Bigr), X_{t-}^{\mathbf{i}(\rho)} \right) =  \left( \Lambda\Bigl(N^K_{t-}, t, (X_{t-}^n)_{n =1}^{N^K_{t-}}, \rho \wedge N^K_{t-}\Bigl) , X_{t-}^{\mathbf{i}(\rho)} \right),
  \]
  is measurable with respect to ${\mathcal{P}}red ( {\mathcal{F}}_t )\otimes {\mathcal{B}}( \RR) $, with ${\mathcal{P}}red ( {\mathcal{F}}_t ) \subseteq  {\mathcal{B}}( \RR)\otimes {\mathcal{F}}  $ the predictable sigma-field associated with  $( {\mathcal{F}}_t )_{t\geq 0}$. 
\end{rmk}

The following identity in law is crucial to check (\hyperlink{c2}{C.2}).

\begin{lem}\label{eq:condlawY}
  Let $(\overline{T}_j)_{j \geq 1}$  denote the sequence of consecutive birth times  in $(0,\infty)$ of  one new particle in the system $ (\nu_t^K)_{t\geq 0}$, constructed with algorithm \textnormal{(\hyperlink{a}{A})}, and  $(T_j,\rho_j)$  be the first two coordinates of the atom $(t,\rho,\theta)$ corresponding to $t=T_j$.  Then, conditionally on ${\mathcal{F}}_{ \overline{T}_j-}$ and $\Bigl\{ \rho_j \leq N^K_{\overline{T}_j-} \Bigr\} $, $Y_{\overline{T}_j}^{\mbox{\scalebox{0.7}{$N^K_{\overline{T}_j}$}}}= \Lambda^{N^K_{\overline{T}_j}} \Bigl( (X_{t-}^n)_{n =1}^{N^K_{t-}},\rho_j \Bigr)$ has law  $\bar{\mu}_{\overline{T}_j} $. 
\end{lem} 
   
\begin{proof}
  Let  $f\colon\RR^d\to \RR$  be a bounded measurable function and $(U_t)_{t\geq 0}$ a  bounded  $( {\mathcal{F}}_t )_{t\geq 0}$-predictable process.  We have
  \begin{multline*}
    f \left(  Y_{\overline{T}_j}^{\mbox{\scalebox{0.7}{$N^K_{\overline{T}_j}$}}}\right) \1_{\bigl\{\rho_j \leq N^K_{\overline{T}_j-} \bigr\} }  U_{\overline{T}_j} \\
    = \int_0^{\infty}   \int_0^{\infty}   \int_0^{\infty}     f\left(    \Lambda^{N^K_{t-}}_t\Bigl((X_{t-}^n)_{n =1}^{N^K_{t-}},\rho\Bigr) \right) \1_{\{\rho \leq N^K_{t-},  \, \overline{N}^K_{t-} = N_0^K  +j -1,  \,  \theta \leq r \} }  U_{t} \, \, {\mathcal{N}}(\dd t, \dd \rho, \dd \theta).
  \end{multline*}
  By Remark \ref{rmk:atoms}, we can use the compensation formula with respect to the filtration $( {\mathcal{F}}_t )_{t\geq 0}$, and deduce with Lemma \ref{lem:optimal_coupling}  that 
  \begin{equation*}
  \begin{split}
    \EE\biggl( f\left(Y_{\overline{T}_j}^{\mbox{\scalebox{0.7}{$N^K_{\overline{T}_j}$}}}\right) \1_{\bigl\{\rho_j \leq N^K_{\overline{T}_j-} \bigr\} }  U_{\overline{T}_j}\biggr)=  & \int_0^{\infty}\int_0^{\infty}     \EE\left(  \langle \bar{\mu}_{t} , f \rangle   N^K_{t}  \1_{\{ \overline{N}^K_{t} = N_0^K  +j -1,  \, \theta \leq r \} }  U_{t}  \right)  \dd \theta \dd t   \\
    = & \,   \EE\left(   \int_{[0, \infty)^3 }   \langle \bar{\mu}_{t} , f \rangle        \1_{\{\rho \leq N^K_{t-},  \, \overline{N}^K_{t-} = N_0^K  +j -1,  \, \theta \leq r \} }  U_{t} \, \,  {\mathcal{N}}(\dd t, \dd \rho, \dd \theta) \right)  \\
    = & \,   \EE\biggl(   \langle \bar{\mu}_{\overline{T}_j} , f \rangle   \1_{\bigl\{\rho_j \leq N^K_{\overline{T}_j-} \bigr\} }  U_{\overline{T}_j}  \biggr) .
  \end{split}
  \end{equation*}
  Since any bounded random variable measurable w.r.t. ${\mathcal{F}}_{ \overline{T}_j-}$ can be written as $U_{\overline{T}_j}$ for some predictable process $(U_t)_{t\geq 0}$, the statement is proved. 
\end{proof}

\begin{prop}\label{prop:condindep} 
  For each $t\geq 0$, conditionally on   $ \langle \nu_t^K, 1\rangle $, the atoms of $ \nu_t^K$   are i.i.d. random variables of law $\bar{\mu}_ t$. 
\end{prop}

\begin{proof} 
  The proof will be  done  constructing  an alternative system  $ (\widehat{\nu}_t^K= \frac{1}{K}\sum_{n=1}^{N^K _t}  \delta_{\widehat{Y}^{n}_t} )_{t\geq 0}$ with  the  same law as    $ (\nu_t^K)_{t\geq 0}$,  for which the required property is easily checked.    This system is defined on the same probability  space as  $ (\nu_t^K)_{t\geq 0}$,    by  means of a variant of the construction of  $ (\nu_t^K)_{t\geq 0}$ in  algorithm (\hyperlink{a}{A}). The algorithm is as follows:  
  \begin{itemize}
    \item[0.] Define for all $j\geq 1$: 
      \[ Z_t^{j}= Z_0^{j}+  \int_0^t b\big(Z_s^{j},H*\mu_s(Z_s^{j})\big) \,\dd s +  \int_0^t\sigma\big(Z_s^{j},G*\mu_s(Z_s^{j})\big) \,\dd W_t^{j} \, , \quad t\geq 0. \]
      Set   $\widehat{Y}_0^n= Z_0^n $ for $n\in \{1,\dots, N^K _0 \} $ and $   \widehat{\nu}^{K}_0 = \frac{1}{K}\sum_{n=1}^{N^K _0}  \delta_{\widehat{Y}_0^n}.$ As before, we set   the same counters  $\overline{N}_0^K=N^K_0$ and $m=0$,  we define $T_0=0$ and  we initialize $j_0(n)=n$ for all $n\geq 1$.
    \item[1.] For $t\geq T_m$, we set $j_t(n)= j_{T_m}(n)$ and $\dd B_t^{n}= \dd W_t^{j_t(n)}$, $n\geq 1$, and we take 
      \begin{equation*} 
        \widehat{Y}_t^{n}= Z_t^{j_t(n)}, \quad  n=1,\dots, N_{T_m}^K,
      \end{equation*}
      until the first time $t>T_m$ with $ (t,\rho,\theta)$   an atom of ${\mathcal{N}}$,  such that $  \rho \leq N_{T_m}^K$ and $ \theta \leq r+ c  N_{T_m}^K /K .$
      We then set $T_{m+1}=t$.
    \item[2.] For $ (t,\rho,\theta)=(T_{m+1},\rho,\theta)$  as before, 
      \begin{itemize}
        \item If $\theta \leq r$, we update $N_t^K \coloneqq N_{t-}^K+1$ and $\overline{N}_t^K \coloneqq \overline{N}_{t-}^K+ 1$, then we define:
          \[\widehat{Y}_t^{N_t^K} \coloneqq  Z_t^{\overline{N}_t^K}.\]
        \item If $r < \theta \leq r + c  N_{T_m}^K /K$, we update $N_t^K \coloneqq N_{t-}^K-1$,  and we redefine:
          \[ \bigl(\widehat{Y}_t^{\mathbf{i}(\rho)}, \widehat{Y}_t^{\mathbf{i}(\rho)+1} ,\dots,  \widehat{Y}_t^{N_t^K} \bigr) \coloneqq  \bigl(\widehat{Y}_{t-}^{\mathbf{i}(\rho)+1}, \widehat{Y}_{t-}^{\mathbf{i}(\rho)+2} ,\dots,  \widehat{Y}_{t-}^{N_{t-}^K} \bigr),\]
        and $j_t(n) \coloneqq j_{t-}(n+1)$ for all $n \geq \mathbf{i}(\rho)$. 
      \end{itemize}
    \item[3.] We increase $m$ by one and go to Step 1.
  \end{itemize}
  Plainly, instead of  sampling  at each  birth time  $\overline{T}_j$  the position of a new independent particle $Y^{\mbox{\scalebox{0.7}{$N^K_{\overline{T}_j}$}}}$   from the atom $(\overline{T}_j,\rho,\theta)$ of  ${\mathcal{N}}$   as in (\hyperlink{a}{A}),   we  now add a new particle $\widehat{Y}^{\mbox{\scalebox{0.7}{$N^K_{\overline{T}_j}$}}}$ to the system by ``turning on'' at that time the nonlinear diffusion process $ Z^{\mbox{\scalebox{0.7}{$\overline{N}^K_{\overline{T}_j}$}}}=Z^{ N_0^K + j} $, which has evolved independently  since time $t=0$,  driven by the same Brownian motion $ W^{ N_0^K + j }$  that drives  the process $\Bigl(Y_t^{\mbox{\scalebox{0.7}{$N^K_{\overline{T}_j}$}}} : t \geq \overline{T}_j\Bigr)$ in the construction (\hyperlink{a}{A}). Call now
  \[ \widehat{{\mathcal{F}}}_t \coloneqq \overline{   \sigma \Bigl(  {\mathcal{F}}_t \vee  \Bigl(Z_{\overline{T}_k}^{N_0^K + k} : N_0^K+k \leq \overline{N}_t^K   \Bigr) \Bigr)  }, \]
  the filtration containing the information effectively employed to construct the process $ (\widehat{\nu}_t^K)$, 
  and let  $(V_t)_{t \geq 0}$ be a bounded  left continuous process adapted to   $( \widehat{{\mathcal{F}}}_t)_{t\geq 0}$. 
  Conditionally on $N_0^K$,   $ V_{\overline{T}_j}$  depends only on ${\mathcal{N}} $ and $(W^k,  Z_0^k) $  for $k< N_0^K + j$,  while  $\bigl(Z_t^{ N_0^K + j}\bigr)_{t\geq 0}$  is independent of them. Therefore, we have
  \begin{equation*}
  \begin{split}
    \EE\biggl( f\biggl(\widehat{Y}_{\overline{T}_j}^{\mbox{\scalebox{0.7}{$N^K_{\overline{T}_j}$}}}\biggr) \1_{\bigl\{\rho_j \leq N^K_{\overline{T}_j-} \bigr\} }  V_{\overline{T}_j} \biggr)= {}&    \EE \biggl( f\Bigl(Z_{\overline{T}_j}^{N^K_0+ j}\Bigr) \1_{\bigl\{\rho_j \leq N^K_{\overline{T}_j-}  \bigr\} }  V_{\overline{T}_j} \biggr)  \\
    = {}&      \EE\biggl(  \langle \bar{\mu}_{\overline{T}_j} , f \rangle \1_{\bigl\{ \rho_j \leq N^K_{\overline{T}_j-} \bigr\} }  V_{\overline{T}_j}\biggr),
  \end{split}
  \end{equation*}
  by Remark \ref{rmk:nonlinearprocess} a).   This implies that, conditionally on  $\widehat{\mathcal{F}}_{ \overline{T}_j-}$ and  $\{ \rho_j \leq N^K_{\overline{T}_j-}\} $, the random variable $\widehat{Y}_{\overline{T}_j}^{\mbox{\scalebox{0.7}{$N^K_{\overline{T}_j}$}}}$ has the law  $\bar{\mu}_{\overline{T}_j} $.  Comparing this to the setting in Lemma \ref{eq:condlawY},   one can check  by induction on $j$ that the processes $ (\nu_t^K)_{t\geq 0}$ and  $ (\widehat{\nu}_t^K)$ have the same law on each of their (common) time intervals $[0, \overline{T}_{j}]$, hence over all $[0,\infty)$. 
 
  To conclude, notice that the i.i.d  processes $(Z_t^j)_{t\geq 0}, j\geq 1$  have law $\bar{\mu}_t$ at each  $t\geq 0$, and they are independent of the filtration $( {\mathcal{G}}_t)_{t\geq 0}  $ with respect to which the process $(N_t^K)_{t\geq 0}$ is measurable.  Moreover,  for each $t\geq 0$, $\{ \widehat{Y}^{1}_t,  \dots, \widehat{Y}_t^{N^K _t} \} = \{ Z^{j_t(1)}_t, \dots,  Z^{j_t(N^K _t)}_t \}  $ is a random subset of  $ \{Z^1_t, \dots,  Z^{  \overline{N}_t^K }_t \} $, selected in a way that is measurable w.r.t. $ {\mathcal{G}}_t $. This readily implies that, conditionally on $N^K _t=N$,  $\{ \widehat{Y}^{1}_t,  \dots, \widehat{Y}_t^{N} \} $ are $N$ i.i.d.  random variables of law $\bar{\mu}_t$, as required.
\end{proof}

\section{Proof of Theorem \ref{thm:rateofconvergence}: pure binary branching case}\label{sec:bd}

We consider in this section the case where interactions take place only through the reproduction events,  that is, due only to the fact  that the position of a newborn individual coincides at its birth with that of its parent (after which all individuals evolve completely independently). We provide the complete proof for this case as it might be of  independent interest, since  convergence bounds are neither available  in this  basic  setting, and also because it is useful to illustrate directly the main arguments.
 
We assume the following throughout this section.

\medskip

\noindent\textbf{Hypothesis (H')}:
\hypertarget{h'}{}
 
\begin{enumerate}
 \item[\hypertarget{h0'}{H.1'.}]  (\hyperlink{h1}{H.1}) holds.
  \item[\hypertarget{h2'}{H.2'.}] The coefficients $\sigma\colon\RR^d \to \RR^{d\otimes d}$ and $b\colon\RR^d\to\RR^d$ do not depend on $\mu_t^K$ and, moreover, they are Lipschitz continuous with $\sigma$ bounded (for simplicity).
  \item[\hypertarget{h3'}{H.3'.}] The individual instantaneous birth and death rates are time inhomogeneous, specified by two measurable functions $r,c\colon[0,T]\to\RR_+$ bounded by some positive constants $\bar{r}$ and $\bar{c}$, respectively.
\end{enumerate}

Notice that, since $r$ and $c$ are deterministic measurable functions of $t$, they are predictable when seen as processes (cf. the sigma-field generated by continuous functions on $\RR_+$ is the Borel sigma-field).

The analog of Theorem \ref{thm:largepopulationlimit} is standard in this scenario (or can be proved by the same techniques used in \cite{FM2015}), and the limit in law of  the process $(\mu_t^K)_{t\geq0}$ is given by the unique weak solution  in ${\mathcal{M}}^+(\RR^d)$ to the linear evolution equation
\begin{equation}\label{eq:linearevol}
  \langle\mu_t,f(t,\cdot)\rangle=\langle\mu_0,f(0,\cdot)\rangle+\int_0^t\langle\mu_s,\partial_sf(s,\cdot)+Lf(s,\cdot)+(r(s)-c(s))f(s,\cdot)\rangle\,\dd s,\quad \forall t\in[0,T],
\end{equation}
for each $f\in C^{1,2}([0,T]\times\RR^d)$, where $L$ is the  time-homogeneous operator defined in \eqref{bd:diffusiongen}.

The construction of the coupling with the auxiliary system is  essentially the same as  in Section \ref{sec:pathwiseconstruction},  using algorithm (\hyperlink{a}{A}) with two minor modifications:
\begin{itemize}
  \item[-] Step 1 is carried out until the first time $t>T_m$, where
    $ (t,\rho,\theta)$ is an atom of ${\mathcal{N}}$ such that $  \rho \leq N_{T_m}^K $  and  $\theta \leq r(t)+ c (t) $, at which one sets $T_{m+1}=t$.
  \item[-]  The updates in Step 2 are carried out according to whether $\theta \leq r(t)$ or otherwise $r(t) < \theta \leq  r(t) + c(t) $.
\end{itemize}
 In  between birth or deaths events,  individuals in the system $(\mu_t^K)_{t\geq 0}$   evolve according to the SDEs
\begin{equation*}
  \dd X_t^n = b(X_t^n)\,\dd t + \sigma(X_t^n)\,\dd B_t^n, \quad n=1,\dots, N_t^K, 
\end{equation*}
as also do the individuals $Y^n$  in the system  $(\nu_t^K)_{t\geq 0}$.

We establish some controls for the mass of the process $(\mu_t^{K})_{t\geq 0}$.  

\begin{lem}\label{lem:boundsmass1}
  For each $T>0$  and $p\geq 1$ there is a constant $C_{T,p}>0$ such that
  \[\sup_{K\in\NN\setminus\{0\}}  \EE\biggl(\sup_{t\in [0,T]}\langle\mu_t^{K},1\rangle^p\biggr)<C_{T,p} \sup_{K\in\NN\setminus\{0\}}\EE(\langle\mu_0^{K},1\rangle^p).\]
  Moreover,  if $ \sup_{K\in\NN\setminus\{0\}}\EE(\langle\mu_0^{K},1\rangle)<\infty$,  for all $T>0$ we have
  \[\EE \big( \big| \langle\mu^K_t,1\rangle - \langle\mu_t,1\rangle \big| \big) \leq C_T  \bigl(  I_1(K) +K^{-\frac{1}{2}} \bigr), \]
  with
 \begin{align*}
  I_1(K)=\EE\bigl(| \langle\mu^K_0,1\rangle-\langle\mu_0,1\rangle|\bigr). 
\end{align*}
\end{lem}

\begin{proof}
  The first claim is shown as in \cite[Lemma 3.3]{FM2015} in a more general setting. For the second assertion, we write the dynamics of the number of particles in the system in terms of the Poisson point measure $\mathcal{N}$ used in algorithm (\hyperlink{a}{A}). We obtain
  \begin{align*}
    N_t^{K}=&\ N_0^{K}+\int_0^t\int_{\mathbb{R}_+}\int_{\mathbb{R}_+}\1_{\rho\leq N_{s-}^{K}}\left(\1_{\theta\leq r(s)}-\1_{r(s)<\theta\leq r(s)+c(s)}\right)\,\mathcal{N}(\dd s,\dd \rho,\dd \theta)\\
    =&\ N_0^{K}+\int_0^t(r(s)-c(s))N_{s}^{K}\,\dd s+M_t^{K},
  \end{align*}
  where $(M_t^K)_{t\geq0}$ is a martingale since,    for all $t\geq 0$, 
  \[ \EE \biggl( \int_0^t\int_{\mathbb{R}_+}\int_{\mathbb{R}_+} \Bigl| \1_{\rho\leq N_{s}^{K}}\Bigl(\1_{\theta\leq r(s)}-\1_{r(s)<\theta\leq r(s)+c(s)}\Big)\Bigr| \,\dd s\dd \rho\dd \theta \biggr) \leq ( \bar{r}+\bar{c} ) \EE \biggl(\int_0^t  N^K_s  \,\dd s \biggr)<\infty, \]
  by the first part and the assumption on the total mass.  Comparing this evolution to the ODE \eqref{eq:linearode} satisfied by the total mass of the limiting measure,  we get the  estimate
  \begin{align*}
    \EE\Big(\Big|\frac{N_t^{K}}{K}-\langle\mu_t,1\rangle\Big|\Big)\leq&\ \EE\Big(\Big|\frac{N_0^{K}}{K}-\langle\mu_0,1\rangle\Big|\Big)+(\bar{r}+\bar{c})\int_0^t\EE\Big(\Big|\frac{N_s^{K}}{K}-\langle\mu_s,1\rangle\Big|\Big)\,\dd s+\EE\Big(\frac{\vert M_t^{K}\vert}{K}\Big). 
  \end{align*}
  The last term is controlled using the Burkholder-Davis-Gundy (BDG) inequality  as follows
  \begin{align*}
    \EE\biggl(\frac{\vert M_t^{K}\vert}{K}\biggr) \leq{}&  \frac{ 1}{K} \,  \EE  \left( \int_0^t\int_{\mathbb{R}_+}\int_{\mathbb{R}_+}\1_{\{ \rho\leq N_{s-}^{K} , \,    \theta\leq r(s)+c(s)\} } \,\mathcal{N}(\dd s,\dd \rho,\dd \theta) \right)^{\frac{1}{2}} \\
    ={}&        \frac{\EE((\int_0^t(r(s)+c(s))N_{s}^{K}\, \dd s) ^{\frac{1}{2}}}{K}   \\ 
      \leq{}& \frac{C_T }{\sqrt{K}}\left( \sup_{K\in\NN\setminus\{0\}}\EE(\langle\mu_0^{K},1\rangle) (\bar{r}+\bar{c})e^{\bar{r}}t\right)^\frac{1}{2}, 
  \end{align*}
  for all $t\in [0,T]$.   We conclude by Gronwall's lemma that
  \begin{align*}
    \EE\Big(\Big|\frac{N_t^{K}}{K}-\langle\mu_t,1\rangle\Big|\Big) \leq C_T \biggl(\EE\Big(\Big|\frac{N_0^{K}}{K}-\langle\mu_0,1\rangle\Big|\Big)+\frac{1}{\sqrt{K}}\biggr).
    \mbox{\qedhere}
  \end{align*}
\end{proof}

The analogue of  Proposition \ref{prop:nonlinearprocess} in this section's setting is rather elementary, yet  illustrative  for the general case,  so we state it in detail and prove it next.

\begin{prop}\label{prop:nonnonlinear}
  Let $(\mu_t)_{t\geq0}$ be the unique weak solution in $\mathcal{M}^+(\RR^d)$ of the linear equation
  \begin{equation*}
    \frac{\partial\mu_t}{\partial t} = L^*\mu_t + (r(t)-c(t))\mu_t,
  \end{equation*}
  with initial condition $\mu_0$ (given as a particular case of Theorem \ref{thm:largepopulationlimit}),  and $(Y_t)_{t\geq 0}$ be the unique pathwise solution to the SDE  \begin{align*}
    Y_t = Y_0 + \int_0^t b(Y_s) \,\dd s + \int_0^t \sigma(Y_s) \,\dd W_s,
  \end{align*}
  where $W$ is a $d$-dimensional Brownian motion and $Y_0$ and independent random variable in $\RR^d$ with law $\bar{\mu}_0$. Then, the flow $(\bar{\mu}_t)_{t\geq0}$   of time-marginal laws   of $(Y_t)_{t\geq 0}$ is the unique weak solution of the Fokker-Planck equation 
  \[ \frac{\partial\bar{\mu}_t}{\partial t} = L^* \bar{\mu}_t, \]
  and  satisfies $\bar{\mu}_t = \mu_t/\langle\mu_t,1\rangle$ for all $t\geq0$.  In particular, for each bounded  real function $f$ we have $\langle\mu_t,f\rangle = \EE(f(Y_t)n_t)$, where $n_t$ is the unique solution with $n_0 = \langle\mu_0,1\rangle$ of the linear differential equation
  \begin{equation}\label{eq:linearode}
    \dd n_t = (r(t)-c(t))n_t \,\dd t. 
  \end{equation}
\end{prop}

\begin{proof}
  The first claim is standard and easily seen using It\^o's formula (uniqueness is also standard using e.g. the Feynman-Kac formula).  The relation between  the law of $Y_t$ and $\mu_t$ for all $t\geq0$ is easily shown considering the function $h(t,x)= \langle\mu_t,1\rangle f(t,x)$ and computing
  \begin{align*}
    \langle\bar{\mu}_t ,h(t, \cdot)\rangle ={}& \langle\bar{\mu}_0,h(0,\cdot)\rangle + \int_0^t \langle\bar{\mu}_s,\partial_sh(s,\cdot)+ Lh(s,\cdot) \rangle \,\dd s\\
    ={}&  \langle \langle\mu_0,1\rangle  \bar{\mu}_0,f(0,\cdot)\rangle + \int_0^t\langle \bar{\mu}_s , f(s,\cdot)\partial_s\langle\mu_s,1\rangle + \langle\mu_s,1\rangle\partial_s f(s,\cdot) + \langle\mu_s,1\rangle L f(s,\cdot)\rangle \,\dd s\\
    ={}&    \langle  \langle\mu_0,1\rangle \bar{\mu}_0,f(0,\cdot)\rangle + \int_0^t\langle \langle\mu_s,1\rangle\bar{\mu}_s , \partial_s f(s,\cdot) + Lf(s,\cdot) + (r(s)-c(s))f(s,\cdot)\rangle \,\dd s. 
  \end{align*}
  This means that $ (\langle\mu_t,1\rangle\bar{\mu}_t )_{t\geq 0}$ satisfies equation  \eqref{eq:linearevol}. 
  Uniqueness for that equation yields $ \langle\mu_t,1\rangle\bar{\mu}_t = \mu_t$ for all $t\geq0$ as claimed.  Consequently, 
    \[ \langle\mu_t,f\rangle = \EE(\langle\mu_t,1\rangle f(Y_t)),\]
  for all bounded $f$, and  the fact that  $(\langle\mu_t,1\rangle)_{t\geq0}$ satisfies  \eqref{eq:linearode} is immediate.
\end{proof}

In order to prove that condition  (\hyperlink{c3}{C.3}) holds,  one last additional estimate is needed, which will be used to  control the joint evolution of coupled particles, in between birth or death events. 

\begin{lem}\label{bd:lem:lipschitzdyn}
  Let $X=(X_t)_{t\geq0}$ and $Y=(Y_t)_{t\geq0}$ be  two  diffusion processes   with generator $L$ driven by the same Brownian motion $B$.  For each $T>0$ there exists $C_T>0$ such that for all $0<u<t<T$
  \begin{align*}
    \EE(\|X_t-Y_t\|^2-\|X_{u}-Y_{u}\|^2)\leq C_T\int_{u}^t\EE(\|X_s-Y_s\|^2) \,\dd s.
  \end{align*}
 \end{lem}

\begin{proof} 
  Let  $(\tau_n)_{n\in\NN}$ be the sequence defined by $ \tau_n \coloneqq \inf\{ s\geq 0 : \| X_s\| ^2 +  \|Y_s \| ^2 >n\}$, which localizes the local martingale parts of $X$ and $Y$.  We first establish a control on the running suprema of the processes.  Using the   fact that $b$ is Lipschitz we obtain
  \begin{multline*}
    \sup_{u\in[0,t\wedge\tau_n]}\|X_u\|^2 \leq 2 \|X_0\|^2 + C_T + C_T\int_0^t \sup_{u\in[0,s\wedge\tau_n]}\|X_u\|^2 \,\dd s \\ 
    + 2\sum_{i,j=1}^d \biggl(\sup_{u\in[0,t\wedge\tau_n]} \biggl|\int_0^u \sigma^{(ij)}(X_s) \,\dd B_s^{(j)}\biggr| \biggr)^2.
  \end{multline*}
  With the BDG inequality and the fact that $\sigma$ is also Lipschitz we then get
  \begin{equation*}
    \EE \biggl(  \sup_{u\in[0,t\wedge\tau_n]}\|X_u\|^2 \biggr)  \leq 2  \EE \left( \|X_0\|^2 \right)+ C_T + C_T\int_0^t   \EE \biggl(   \sup_{u\in[0,s\wedge\tau_n]}\|X_u\|^2  \biggr) \,\dd s . 
  \end{equation*}
  Applying Gronwall's lemma and then Fatou's lemma upon letting $n\to \infty$ we deduce
  \begin{align}\label{eq:controlsup2}
    \EE\biggl( \sup_{s\in[0,T]}\|X_t\|^2\biggr) \leq C_T(\EE(\|X_0\|^2) + 1),
  \end{align}
  and a similar estimate  holds for the process $Y$.  Now,   Itô's formula shows that
  \begin{multline*}
    \|X_t-Y_t\|^2 = \|X_u-Y_u\|^2 + \int_u^t2 (X_s-Y_s)^{\mathrm{t}}(b(X_s)-b(Y_s)) \,\dd s\\
    + \int_u^t 2(X_s-Y_s)^{\mathrm{t}}(\sigma(X_s)-\sigma(Y_s))\ \dd B_s +\sum_{i,j=1}^d\int_u^t(\sigma^{(ij)}(X_s)-\sigma^{(ij)}(Y_s))^2 \,\dd s.
  \end{multline*}
  The sequence $(\tau_n)_n$  localizes the local martingale on the right hand side. Taking expectation for the stopped process and using the  Lipschitz  character of  $b$ and $\sigma$  leads to
  \[\EE(\|X_{t\wedge \tau_n}-Y_{t\wedge \tau_n}\|^2)\leq \EE(\|X_{u}-Y_{u}\|^2)+C\int_{u}^{t}\EE(\|X_{s\wedge\tau_n}-Y_{s\wedge\tau_n}\|^2)\,\dd s.\]
  By dominated convergence using the bound \eqref{eq:controlsup2},  we can take  $n\to\infty$ and conclude. 
\end{proof}

Now we can state the bound leading to condition  (\hyperlink{c3}{C.3}) and  to the proof of the main result,  in the case of pure binary branching.  

\begin{lem}\label{bd:lem:orig_indepdistance} 
  There exists a constant $C_T>0$ depending on $d$ and $q$,  such that for all $K\in\NN\setminus\{0\}$ and $t\in [0,T]$: 
  \begin{equation*}
    \EE\biggl(\frac{1}{K}\sum_{n=1}^{N_t^{K}}\left\| X_t^n-Y_t^n\right\|^2\biggr) \leq C_T \int_0^t \EE\biggl(\frac{N_s^K}{K}W_2^2(\bar{\nu}_s^K,\bar{\mu}_s)\biggr) \,\dd s.
  \end{equation*}
\end{lem}

\begin{proof}
  Consider the product empirical measure $\eta_t^K\coloneqq\frac{1}{K}\sum_{n=1}^{N_t^{K}}\delta_{(X_t^n,Y_t^n)}$ and the sequence of jump times $(T_m)_{m\in\mathbb{N}}$ of the process $(N_t^K)_{t\geq 0}$,  defined through algorithm (\hyperlink{a}{A}). We  decompose the evolution of $\eta_t^K$ in terms of  $(T_m)_{m\in\mathbb{N}}$ as  follows
  \begin{align*}
    \eta_t^K=\eta_t^K+\sum_{m=1}^\infty\bigl(\1_{t\geq T_m}\bigl(\eta_{T_m}^K-\eta_{T_{m-1}}^K\bigr)-\1_{T_{m+1}>t>T_m}\eta_{T_m}^K\bigr)+\eta_0^K.
  \end{align*}
  Defining $A_t^K\coloneqq\sum_{s\leq t}\vert \Delta N_s^{K}\vert$, where $\Delta N_s^{K}=N_s^{K}-N_{s-}^{K}$, we can rewrite the previous equality as
  \begin{align*}
    \eta_t^K    =&\ \eta_0^K+\eta_t^K-\eta_{T_{A_t^K}}^K+\sum_{m=1}^\infty \1_{ t\geq T_m }\left(\eta_{T_m}^K-\eta_{T_m^{-}}^K+\eta_{T_m^{-}}^K-\eta_{T_{m-1}}^K\right).
  \end{align*}
  The aim of this  decomposition is to control  separately what happens in between jumps and at the jump instants.
  Integrating the function $d_2(x,y) \coloneqq \| x-y\|^2$ and taking expectation yields
  \begin{equation}\label{proof:jumpdecomp}
    \begin{split}
    \EE\bigl(\bigl\langle\eta_t^K,d_2\bigr\rangle\bigr) &= \mathbb{E}\bigl(\langle\eta_0^K,d_2\rangle\bigr)+\mathbb{E}\bigg(\sum_{m=1}^\infty \1_{t\geq T_m}\Big(\langle\eta_{T_m}^K,d_2\rangle-\langle\eta_{T_m^{-}}^K,d_2\rangle\Big)\bigg) \\
    &\qquad +\mathbb{E}\bigg(\bigl\langle\eta_t^K,d_2\bigr\rangle-\Big\langle\eta_{T_{A_t^K}}^K,d_2\Big\rangle+\sum_{m=1}^\infty \1_{t\geq T_m}\left(\langle\eta_{T_m^{-}}^K,d_2\rangle-\langle\eta_{T_{m-1}}^K,d_2\rangle\right)\bigg).
    \end{split}
  \end{equation}
  By Lemma \ref{bd:lem:lipschitzdyn}, and since the evolution of $\eta_t^K$ is independent of the sigma-field  $(\mathcal{G}_t)_{t\geq0}$  on each interval $[T_{m-1}, T_m)$, we get
  \begin{align}\label{proof:betweenjumps}
    \mathbb{E}\bigl( \1_{t\geq T_m} \hspace{-0.25ex} (\langle\eta_{T_m^{-}}^K,d_2\rangle \hspace{-0.3ex} - \hspace{-0.3ex} \langle\eta_{T_{m-1}}^K,d_2\rangle\bigr) \,\big\vert\, \mathcal{G}_t \bigr) ={}& \mathbb{E}\biggl( \frac{1}{K} \hspace{-0.5ex} \sum_{n=1}^{N_{T_{m-1}}^{K}} \hspace{-0.7ex} \| X_{T_m^{-}}^n \hspace{-0.5ex} - \hspace{-0.15ex} Y_{T_m^{-}}^n\|^2 \hspace{-0.3ex} - \hspace{-0.3ex} \| X_{T_{m-1}}^n \hspace{-0.5ex} - \hspace{-0.15ex} Y_{T_{m-1}}^n\|^2 \,\bigg\vert\, \mathcal{G}_t \biggr) \1_{t\geq T_m} \nonumber\\
    \leq{}& \frac{1}{K}\sum_{n=1}^{N_{T_{m-1}}^{K}}C\int_{T_{m-1}}^{T_m^{-}}\mathbb{E}\bigl(\| X_s^n-Y_s^n\|^2\bigm\vert \mathcal{G}_t\bigr) \,\dd s  \1_{t\geq T_m}  \nonumber\\
    ={}& C\int_{T_{m-1}}^{T_m^{-}}\mathbb{E}\bigl(\langle\eta_s^K,d_2\rangle\bigm\vert \mathcal{G}_t\bigr) \,\dd s  \1_{t\geq T_m}, 
  \end{align}
  and similarly,  for the remaining time interval, 
  \begin{align*}
    \mathbb{E}\Bigl(\mathbb{E}\Bigl(\bigl\langle\eta_t^K,d_2\bigr\rangle-\big\langle\eta_{T_{\scalebox{0.6}{$A_t^K$}}}^K,d_2\big\rangle\Bigm\vert \mathcal{G}_t \Bigr)\Bigr)\leq C\int_{T_{\scalebox{0.6}{$A_t^K$}}}^t\mathbb{E}\bigl(\langle\eta_s^K,d_2\rangle\bigm\vert \mathcal{G}_t\bigr) \,\dd s.
  \end{align*}
  Recalling  Step 2 of the variant of algorithm (\hyperlink{a}{A}) used in this section,  the term involving the jumps of the processes can be written as
  \begin{align}\label{proof:jumps}
    \mathbb{E}\biggl(\sum_{m=1}^\infty & \1_{t\geq T_m} \left( \langle\eta_{T_n}^K,d_2\rangle - \langle\eta_{T_n^{-}}^K,d_2\rangle \right)\biggr) \notag\\
    &\hspace{2ex}= \mathbb{E}\Bigl(\frac{1}{K}\int_{[0,t]\times\RR_+\times\RR_+} \Bigl(\1_{\rho\leq N_{s-}^{K}} \1_{\theta\leq r(s)} \Bigl\| X_s^{N_{s}^{K}} - Y_s^{N_s^K} \Bigr\|^2 \notag\\
    &\hspace{2ex}\qquad- \1_{\rho\leq N_{s-}^{K}} \1_{r(s)<\theta\leq r(s)+c(s)} \Bigl\| X_{s-}^{\ii(\rho)}-Y_{s-}^{\ii(\rho)}\Bigr\|^2\Bigr)\,\mathcal{N}(\dd s,\dd \rho,\dd \theta)\Bigr) \notag\\
    &\hspace{2ex} \leq \mathbb{E}\Bigl( \int_{[0,t]\times\RR_+\times\RR_+}\hspace{-1ex}  \frac{1}{K} \1_{\rho\leq N_{s-}^{K}}\1_{\theta\leq r(s)} \Bigl\| X_{s-}^{\ii(\rho)} - \Lambda_{s}^{N_{s-}^K}\Big( (X_{s-}^n)_{n=1}^{N_{s-}^K} ,\rho\Big)\Bigr\|^2\,\mathcal{N}(\dd s,\dd \rho,\dd \theta)\Bigr) \notag\\
    &\hspace{2ex}= \mathbb{E}\Bigl( \int_0^t \hspace{-1ex}  \, \frac{N_{s}^K}{K} \, r(s)W_2^2\bigl(\bar{\mu}_s^K,\bar{\mu}_s\bigr)\,\dd s \Bigr),
  \end{align}
  where we used  Lemma \ref{lem:optimal_coupling} and Remark \ref{rmk:atoms} in the last equality. Since $\EE(\langle\eta_0^K,d_2\rangle)=0$,  combining the two previous estimates  and writing $C$ for some constant that may change from line to line, we deduce
  \begin{align*}
    \EE(\langle\eta_t^K,d_2\rangle) \leq{}& C \int_0^t\EE(\langle\eta_s^K,d_2\rangle) \,\dd s + \EE\biggl(\int_0^t\frac{N_{s}^{K}}{K}r(s)W_2^2\bigl(\bar{\mu}_s^K,\bar{\mu}_s\bigr) \,\dd s\biggr)\\
    \leq{}& C \int_0^t \EE(\langle\eta_s^K,d_2\rangle) \,\dd s + C \int_0^t \EE\Bigl(\frac{N_{s}^{K}}{K}W_2^2\bigl(\bar{\nu}_s^K,\bar{\mu}_s\bigr)\Bigr) \,\dd s \\
    &\qquad + C \int_0^t \EE\Bigl(\frac{N_{s}^{K}}{K}W_2^2\bigl(\bar{\mu}_s^K,\bar{\nu}_s^K\bigr)\Bigr) \,\dd s\\
    \leq{}& C \int_0^t \EE(\langle\eta_s^K,d_2\rangle) \,\dd s + C \int_0^t \EE\Bigl(\frac{N_{s}^{K}}{K}W_2^2\bigl(\bar{\nu}_s^K,\bar{\mu}_s\bigr)\Bigr) \,\dd s,
  \end{align*}
  where in the last inequality, we used the fact that
  \begin{align}\label{eq:NKW_leq_eta}
    \EE\Big(\frac{N_t^{K}}{K}W_2^2\bigl(\bar{\mu}_t^K,\bar{\nu}_t^K\bigr)\Big)\leq\EE\biggl(\frac{1}{K}\sum_{n=1}^{N_t^{K}}\| X_t^n-Y_t^n\|^2\biggr),
  \end{align}
  since $W_2^2\bigl(\bar{\mu}_t^K,\bar{\nu}_t^K\bigr)\leq  \frac{1}{N_t^{K}}\sum_{n=1}^{N_t^{K}}\| X_t^n-Y_t^n\|^2$.   We conclude by Gronwall's lemma.
\end{proof}

\begin{cor}\label{cor:improvedC3}
  Condition \textnormal{(\hyperlink{c3}{C.3})} holds under  \textnormal{(\hyperlink{h'}{H'})} with the improved bound:  $C_T  R_{d,q}(K)$. 
\end{cor}

\begin{proof}
  Combine  inequality \eqref{eq:NKW_leq_eta} with Lemma  \ref{bd:lem:orig_indepdistance} and apply then Lemma \ref{lem:boundNW}. 
\end{proof}

We now have everything that is needed  to prove our main result in the case of pure branching diffusions.

\begin{proof}[Proof of Theorem \ref{thm:rateofconvergence} under \textnormal{(\hyperlink{h'}{H'})}]
  Since $\langle \mu_t^K,1\rangle= N_t^{K}/K $, applying Lemma \ref{lem:normwasserstein} and the triangle inequality for $W_1$ we  get 
  \begin{align}\label{eq:boundproofthm2}
    \EE\left(  \Vert\mu_t^K-\mu_t\Vert_ {{\BL}^*}  \right) &\leq    \EE\Bigl(\frac{N_t^K}{K}W_1\bigl(\bar{\nu}_t^K, \bar{\mu}_t^K \bigr)\Bigr) +    \EE\Bigl(\frac{N_t^K}{K}W_1\bigl(\bar{\nu_t}^K, \bar{\mu_t} \bigr)\Bigr)  + \EE \bigl(\bigl| \langle\mu^K_t,1\rangle - \langle\mu_t,1\rangle\bigr|\bigr) \notag  \\
    &\leq \biggl( \EE\Big(\frac{N_t^{K}}{K}W_2^2\bigl(\bar{\nu}_t^K,\bar{\mu}_t\bigr)\Big)^\frac{1}{2} +  \EE\Big(\frac{N_t^{K}}{K}W_2^2\bigl(\bar{\nu}_t^K,\bar{\mu}_t^K\bigr)\Big)^\frac{1}{2}  \biggr)  \EE\Big(\frac{N_t^{K}}{K} \Big)^\frac{1}{2} \notag \\
    &\qquad + \EE \bigl(\bigl| \langle\mu^K_t,1\rangle - \langle\mu_t,1\rangle\bigr|\bigr),
  \end{align}
  where we also used the Cauchy-Schwarz inequality and  the inequality  $W^2_1\leq W^2_2$ in the second line.
  Thanks to Lemma \ref{lem:boundsmass1} we obtain
  \[
    \EE\left(  \Vert\mu_t^K-\mu_t\Vert_ {{\BL}^*}  \right) \leq C_T  \biggl( \EE\Big(\frac{N_t^{K}}{K}W_2^2\bigl(\bar{\nu}_t^K,\bar{\mu}_t\bigr)\Big)^\frac{1}{2} +  \EE\Big(\frac{N_t^{K}}{K}W_2^2\bigl(\bar{\nu}_t^K,\bar{\mu}_t^K\bigr)\Big)^\frac{1}{2} +  I_1(K) +K^{-1/2}  \biggr).
  \]
  Now, thanks to the first bound in Lemma \ref{lem:boundsmass1}, the uniform moment control stated in Remark \ref{rmk:nonlinearprocess} b), and  conditions (\hyperlink{c1}{C.1}) and (\hyperlink{c2}{C.2}), we can apply  Lemma \ref{lem:boundNW} to  $\bar{\nu}= \bar{\nu}^K_t$, $N=N_t^K$, and $\bar{\mu}=\bar{\mu}_t$ to bound the first term in the right hand side by $R^{1/2}_{d,q}(K)$. The second term is bounded   by $C_T R^{1/2}_{d,q}(K)$, due to Corollary \ref{cor:improvedC3}. Since $ K^{-1/2} \leq R_{d,q}^{1/2}$, we conclude that
  \[
    \EE\left(  \Vert\mu_t^K-\mu_t\Vert_ {{\BL}^*}  \right) \leq C_T \Bigl( R_{d,q}(K)^{\frac{1}{2}} + I_1(K) \Bigr).
  \]
\end{proof}

\section{Proof of Theorem \ref{thm:rateofconvergence}: general case}\label{sec:bdi}

We now consider processes $(\mu_t^K)_{t\geq0}$ satisfying the  general assumptions of Theorem \ref{thm:rateofconvergence}.   We start by establishing bounds for the mass of the process,  analogous to the bounds   in Lemma \ref{lem:boundsmass1}. The convergence bound is less sharp and more difficult to establish now because of the nonlinearities coming from  the interaction.

\begin{lem}\label{lem:boundsmass}
  For each $T>0$  and $p\geq 1$ there is a constant $C_{T,p}>0$ such that  
  \[\sup_{K\in\NN\setminus\{0\}}  \EE\biggl(\sup_{t\in [0,T]}\langle\mu_t^{K},1\rangle^p\biggr)<C_{T,p} \sup_{K\in\NN\setminus\{0\}}\EE(\langle\mu_0^{K},1\rangle^p).\]
  Moreover, if   $\sup_{K\in\NN\setminus\{0\}}\EE(\langle\mu_0^{K},1\rangle^4)<\infty$,  for all $T>0$ we have
  \[ \EE \Bigl( \bigl(  \langle\mu^K_t,1\rangle - \langle\mu_t,1\rangle\bigr)^4\Bigr)\leq C_T  \bigr(  I_4^4(K) +K^{-1} \bigl) .\]
\end{lem}

\begin{proof}
  For the first bound on the moments of the total mass we refer to \cite[Lemma 3.3]{FM2015}.  To prove the convergence bound in the second part,  we resort to algorithm (\hyperlink{a}{A}) to represent the dynamics of the number of particles by the SDE
  \begin{align*}
    N_t^K ={}& N_0^K + \int_0^t\int \1_{\rho\leq N_{s-}^K} \bigg(\1_{\theta\leq r} - \1_{r < \theta \leq r+c\frac{N_{\scalebox{0.5}{$s-$}}^{\scalebox{0.5}{$K$}}}{K}}\bigg)\ \mathcal{N}(\dd s,\dd\rho, \dd\theta) \\
    ={}& N_0^K + \int_0^t \bigg(r-c\frac{N_s^K}{K}\bigg)N_s^K \,\dd s + M_t^K. 
  \end{align*}
  Notice that the process $(M_t^K)_{t\geq0}$ is a  martingale since,    for all $t\geq 0$, 
  \[ \EE \biggl(\int_0^t\int_{\mathbb{R}_+}\int_{\mathbb{R}_+} \biggl| \1_{\rho\leq N_{s}^{K}}\biggl(\1_{\theta\leq r}-\1_{r < \theta \leq r+c\frac{N_{\scalebox{0.5}{$s-$}}^{\scalebox{0.5}{$K$}}}{K}}\biggr)\biggr| \,\dd s\, \dd \rho  \, \dd \theta \biggr) \leq ( r+c )\EE \biggl( \int_0^t  (N^K_s)^2  \,\dd s \biggr)<\infty, \]
  by the previous part and the assumptions on the total mass of the system. The limiting mass in turn satisfies the dynamics
  \begin{align*}
    \langle\mu_t,1\rangle = \langle\mu_0,1\rangle + \int_0^t (r - c\langle\mu_s,1\rangle)\langle\mu_s,1\rangle \,\dd s. 
  \end{align*}
  We will first establish an $L^2$ convergence bound for the total mass. Using It\^o's formula  we get 
  \begin{multline*}
    \bigg( \frac{N_t^{K}}{K}-\langle\mu_t,1\rangle \bigg)^2 = \bigg( \frac{N_0^{K}}{K}-\langle\mu_0,1\rangle \bigg)^2 + \int_0^t 2 \bigg(\frac{N_{s-}^K}{K}-\langle\mu_{s-},1\rangle\bigg)\ \dd \Bigl(\frac{M_s^K}{K}\Bigr)\\
    \hspace{22ex} + \int_0^t \bigg[2r\bigg( \frac{N_s^{K}}{K}-\langle\mu_s,1\rangle \bigg)^2 - \bigg( \frac{N_s^{K}}{K}-\langle\mu_s,1\rangle \bigg)^2\bigg( \frac{N_s^{K}}{K}+\langle\mu_s,1\rangle \bigg) \bigg] \,\dd s \\
    \hspace{30ex} + \int_0^t\int \1_{\rho\leq N_{s-}^K}\1_{r < \theta \leq r+c\frac{N_{\scalebox{0.5}{$s-$}}^{\scalebox{0.5}{$K$}}}{K}} \bigg(\frac{1}{K}\bigg)^2 \,\mathcal{N}(\dd s,\dd\rho, \dd\theta) \\
    + \int_0^t\int \1_{\rho\leq N_{s-}^K} \1_{\theta\leq r} \bigg(\frac{1}{K}\bigg)^2 \,\mathcal{N}(\dd s,\dd\rho, \dd\theta). 
  \end{multline*}
  Bounding above the negative term in the second line by 0 gives us 
  \begin{align}\label{eq:boundsquareevol}
    \hspace{-0.7em}\bigg( \frac{N_t^{K}}{K}-\langle\mu_t,1\rangle \bigg)^2 &\leq \bigg( \frac{N_0^{K}}{K}-\langle\mu_0,1\rangle \bigg)^2 + \int_0^t 2r\bigg( \frac{N_s^{K}}{K}-\langle\mu_s,1\rangle \bigg)^2 \,\dd s + \int_0^t \frac{r}{K}\biggl(\frac{N_s^K}{K}\biggr) \,\dd s \notag\\
    &\quad\ \ + \int_0^t \frac{c}{K}\biggl(\frac{N_s^K}{K}\biggr)^2 \,\dd s  + \int_0^t 2 \bigg(\frac{N_{s-}^K}{K}-\langle\mu_{s-},1\rangle\bigg) \,\dd \Bigl(\frac{M_s^K}{K}\Bigr) + \bar{M}_t^K + \tilde{M}_t^K  \notag\\
    &\leq \bigg( \frac{N_0^{K}}{K}-\langle\mu_0,1\rangle \bigg)^2 + \int_0^t 2r\bigg( \frac{N_s^{K}}{K}-\langle\mu_s,1\rangle \bigg)^2 \,\dd s + \frac{rT}{K}\sup_{s\in[0,T]}\langle\mu_s^K,1\rangle \notag \\
    &\quad\ \ + \frac{cT}{K}\sup_{s\in[0,T]}\langle\mu_s^K,1\rangle^2 + \int_0^t 2 \bigg(\frac{N_{s-}^K}{K}-\langle\mu_{s-},1\rangle\bigg) \,\dd \Bigl(\frac{M_s^K}{K}\Bigr) + \bar{M}_t^K + \tilde{M}_t^K,
  \end{align}
  where $(\bar{M}_t^K)_{t\geq0}$ and $(\tilde{M}_t^K)_{t\geq0}$ are compensated Poisson integrals. Let now  $(\tau_m)_m$ be the sequence of stopping times defined by $\tau_m = \inf\{t>0 : \overline{N}_t^K >m \} $ for $m\geq1$ and $\tau_0=0$. Since $\overline{N}_{s}^K$ is increasing  by one and $\overline{N}_{r-}^K\geq m+1 =  \overline{N}_{\tau_m}^K> \overline{N}_{s-}^K $ for all $r>\tau_m\geq s$,  we have
    \begin{align*}
    \int_0^{t\wedge\tau_m} 2 \bigg(\frac{N_{s-}^K}{K}-\langle\mu_{s-},1\rangle\bigg) \,\dd \Bigl(\frac{M_s^K}{K}\Bigr) 
    ={}& 2 \int_0^{t} \1_{\{\overline{N}_{s-}^K\leq m\}} \bigg(\frac{N_{s-}^K}{K}-\langle\mu_{s-},1\rangle\bigg) \,\dd \Bigl(\frac{M_s^K}{K}\Bigr)  \\
     ={}& 2 \int_0^t\int \phi(s,\rho,\theta) \tilde{\mathcal{N}}(\dd s,\dd\rho,\dd\theta),
  \end{align*}
  with $\tilde{\mathcal{N}}$  the compensated measure associated with $\mathcal{N}$ and $\phi$ the predictable process
  \begin{align*}
    \phi(s,\rho,\theta) = \1_{\overline{N}_{s-}^K\leq m}\1_{\rho\leq N_{s-}^K} \frac{1}{K}\bigg(\1_{\theta\leq r} - \1_{r < \theta \leq r+c\frac{N_{\scalebox{0.5}{$s-$}}^{\scalebox{0.5}{$K$}}}{K}}\bigg)\bigg(\frac{N_{s-}^K}{K}-\langle\mu_{s-},1\rangle\bigg).
  \end{align*}
  The inequality  $ N_s^K\leq \overline{N}_{s}^K$ implies that
    \begin{align*}
    \EE\biggl( \int_0^t\int_0^{\infty} \int_0^{\infty}  |\phi(s,\rho,\theta)| \,\dd s \, \dd\rho \, \dd\theta \biggr) \leq{}& \EE\biggl(  \int_0^t \1_{\overline{N}_{s}^K\leq m}(s) \frac{N_s^K}{K}\Bigl(r+c\frac{N_s^K}{K}\Bigr)\Bigl(\frac{N_s^K}{K}+\langle\mu_s,1\rangle\Bigr) \,\dd s \biggr) \\
    \leq{}& \EE\biggl( \int_0^t  \frac{m}{K}\Bigl(r+c\frac{m}{K}\Bigr)\Bigl(\frac{m}{K}+\langle\mu_s,1\rangle\Bigr) \,\dd s \biggr) \\
    \leq{}& C_{T,K,m} \biggl( 1+\sup_{s\in[0,T]}\langle\mu_s,1\rangle \biggr),
  \end{align*}
  and so the integral w.r.t.  $\dd \Bigl(\frac{M_s^K}{K}\Bigr) $ in \eqref{eq:boundsquareevol} is a martingale. By  similar reasonings, the stopped processes  $(\bar{M}_{t\wedge\tau_m} ^K)_{t\geq0}$ and $(\tilde{M}_{t\wedge\tau_m}^K)_{t\geq0}$ are also seen to be martingales. Taking expectation in  \eqref{eq:boundsquareevol} we get 
  \begin{align*}
    \EE\biggl(\hspace{-0.2ex}\bigg( \frac{N_{t\wedge\tau_m}^{K}}{K}-\langle\mu_{t\wedge\tau_m},1\rangle \bigg)^{\hspace*{-0.5ex}2}\biggr) \leq{}& \EE\biggl(\hspace{-0.2ex}\bigg( \frac{N_0^{K}}{K}-\langle\mu_0,1\rangle \bigg)^{\hspace*{-0.5ex}2}\biggr) + \EE\biggl(\int_0^{t\wedge\tau_m}\hspace{-0.5ex} 2r\bigg( \frac{N_s^{K}}{K}-\langle\mu_s,1\rangle \bigg)^{\hspace*{-0.5ex}2} \,\dd s\biggr) + \frac{C_T}{K} \\
    \leq{}&\EE\biggl(\hspace{-0.2ex}\bigg( \frac{N_0^{K}}{K}-\langle\mu_0,1\rangle \bigg)^{\hspace*{-0.5ex}2}\biggr) + \int_0^{t}\hspace{-0.5ex} 2r\EE\biggl(\hspace{-0.2ex}\bigg( \frac{N_{s\wedge\tau_m}^{K}}{K}-\langle\mu_{s\wedge\tau_m},1\rangle \bigg)^{\hspace*{-0.5ex}2}\biggr) \,\dd s + \frac{C_T}{K}.
  \end{align*}
  Using Gronwall's lemma we obtain
  \begin{equation}\label{eq:stopgronwall}
    \EE\biggl(\bigg( \frac{N_{t\wedge\tau_m}^{K}}{K}-\langle\mu_{t\wedge\tau_m},1\rangle \bigg)^2\biggr) \leq \biggl(\EE\biggl(\bigg( \frac{N_0^{K}}{K}-\langle\mu_0,1\rangle \bigg)^2\biggr) + \frac{C_T}{K} \biggr)e^{2rT}. 
  \end{equation}
  By Fatou's lemma,  we then get $\EE \Bigl( \bigl(  \langle\mu^K_t,1\rangle - \langle\mu_t,1\rangle\bigr)^2\Bigr)\leq C_T  \bigr(  I_2^2(K) +K^{-1} \bigl)$, but the bound \eqref{eq:stopgronwall} will be more practical for our purposes.
  Let us now address the $L^4$ bound.  Applying It\^o's formula again we get
   \begin{align*}
    \bigg( \frac{N_t^K}{K} - \langle\mu_t,1\rangle \bigg)^4 ={}& \bigg( \frac{N_0^{K}}{K}-\langle\mu_0,1\rangle \bigg)^4 + \int_0^t4\bigg(\frac{N_{s^{-}}^K}{K} - \langle\mu_{s^{-}},1\rangle\bigg)^3 \,\dd \biggl(\frac{M_s^K}{K}\biggr) \\
    &{} + \int_0^t \bigg[4r\bigg( \frac{N_s^{K}}{K}-\langle\mu_s,1\rangle \bigg)^4 - 4\bigg( \frac{N_s^{K}}{K}-\langle\mu_s,1\rangle \bigg)^4\bigg( \frac{N_s^{K}}{K}+\langle\mu_s,1\rangle \bigg) \bigg] \,\dd s \\
    {}&\hspace{5ex} + \int_0^t\int \1_{\rho\leq N_{s-}^K} \1_{\theta\leq r} \bigg[ \bigg( \frac{N_{s-}^{K}}{K}-\langle\mu_{s-},1\rangle +\frac{1}{K}\bigg)^4 \\
    {}& \hspace{10ex}- \bigg( \frac{N_{s-}^{K}}{K}-\langle\mu_{s-},1\rangle \bigg)^4
    - 4\bigg( \frac{N_{s-}^{K}}{K}-\langle\mu_{s-},1\rangle \bigg)^3\frac{1}{K}\bigg] \mathcal{N}(\dd s,\dd\rho, \dd\theta) \\
    {}&\hspace{5ex} + \int_0^t\int \1_{\rho\leq N_{s-}^K}\1_{r < \theta \leq r+c\frac{N_{\scalebox{0.5}{$s-$}}^{\scalebox{0.5}{$K$}}}{K}}  \bigg[\bigg( \frac{N_{s-}^{K}}{K}-\langle\mu_{s-},1\rangle -\frac{1}{K}\bigg)^4 \\
    {}& \hspace{10ex} - \bigg( \frac{N_{s-}^{K}}{K}-\langle\mu_{s-},1\rangle \bigg)^4 + 4\bigg( \frac{N_{s-}^{K}}{K}-\langle\mu_{s-},1\rangle \bigg)^3\frac{1}{K} \bigg] \mathcal{N}(\dd s,\dd\rho, \dd\theta).
  \end{align*}
  Bounding the negative term in the second line by 0 and compensating the Poisson integrals gives us  \begin{align*}
    \bigg( \frac{N_t^K}{K} - \langle\mu_t,1\rangle \bigg)^4 \leq{}& \bigg( \frac{N_0^{K}}{K}-\langle\mu_0,1\rangle \bigg)^4 + \int_0^t 4r\bigg( \frac{N_s^{K}}{K}-\langle\mu_s,1\rangle \bigg)^4 \,\dd s \\
    {}& + \int_0^{t} r N_s^K \bigg( 6\bigg( \frac{N_s^{K}}{K}-\langle\mu_s,1\rangle \bigg)^2\frac{1}{K^2} + 4\bigg( \frac{N_s^{K}}{K}-\langle\mu_s,1\rangle \bigg)\frac{1}{K^3} + \frac{1}{K^4} \bigg) \,\dd s \\
    {}&\hspace{0.5ex} + \int_0^{t} cN_s^K\frac{N_s^K}{K}\bigg( 6\bigg( \frac{N_s^{K}}{K}-\langle\mu_s,1\rangle \bigg)^2\frac{1}{K^2} - 4\bigg( \frac{N_s^{K}}{K}-\langle\mu_s,1\rangle \bigg)\frac{1}{K^3} + \frac{1}{K^4} \bigg) \,\dd s \\
    {}&\hspace{27ex} + \int_0^t4\bigg(\frac{N_{s^{-}}^K}{K} - \langle\mu_{s^{-}},1\rangle\bigg)^3 \,\dd \biggl(\frac{M_s^K}{K}\biggr) + R_t^K + \bar{R}_t^K,
  \end{align*}
  where $(R_t^K)_{t\geq0}$ and $(\bar{R}_t^K)_{t\geq0}$ are compensated Poisson integrals. Using Young's inequality we deduce that
  \begin{multline}\label{eq:boundfourthevol}
    \biggl( \frac{N_t^K}{K} - \langle\mu_t,1\rangle \biggr)^4 \leq \biggl( \frac{N_0^{K}}{K}-\langle\mu_0,1\rangle \biggr)^4 \hspace{-0.3ex}+ C\int_0^t \biggl( \frac{N_s^{K}}{K}-\langle\mu_s,1\rangle \biggr)^4 \,\dd s + \frac{C}{K^2} \int_0^t \biggl( \frac{N_s^{K}}{K}-\langle\mu_s,1\rangle \biggr)^2 \,\dd s \\
    \hspace{25ex} + \frac{C_T}{K^3}\sup_{s\in[0,T]}\langle\mu_s^K,1\rangle + \frac{C_T}{K}\sup_{s\in[0,T]}\langle\mu_s^K,1\rangle^2 + \frac{C_T}{K}\sup_{s\in[0,T]}\langle\mu_s^K,1\rangle^4 \\
    + \int_0^t4\biggl(\frac{N_{s^{-}}^K}{K} - \langle\mu_{s^{-}},1\rangle\biggr)^3 \,\dd \biggl(\frac{M_s^K}{K}\biggr) + R_t^K + \bar{R}_t^K.
  \end{multline}
  Proceeding in a similar way as in the proof of the bound  \eqref{eq:stopgronwall}, we can verify again  that the three processes in the last line are martingales if stopped  at $\tau_m = \inf\{t>0 : \overline{N}_t^K >m \}$. Thus, stopping the inequality \eqref{eq:boundfourthevol} and  taking expectation yields
  \begin{equation*}
    \begin{split}
    \EE\biggl(\bigg( \frac{N_{t\wedge\tau_m}^K}{K} - \langle\mu_{t\wedge\tau_m},1\rangle \bigg)^4\biggr)  &\leq I_4^4(K) + \frac{C_T}{K} + C\int_0^{t} \EE\biggl(\bigg( \frac{N_{s\wedge\tau_m}^{K}}{K}-\langle\mu_{s\wedge\tau_m},1\rangle \bigg)^4\biggr) \,\dd s \\
    &\qquad + \frac{C}{K^2} \int_0^{t} \EE\biggl(\bigg( \frac{N_{s\wedge\tau_m}^{K}}{K}-\langle\mu_{s\wedge\tau_m},1\rangle \bigg)^2\biggr) \,\dd s \\
    &\leq I_4^4(K) + \frac{C_T}{K} + C\int_0^{t} \EE\biggl(\bigg( \frac{N_{s\wedge\tau_m}^{K}}{K}-\langle\mu_{s\wedge\tau_m},1\rangle \bigg)^4\biggr) \,\dd s \\
    &\qquad + \frac{C_TT}{K^2}\Bigl(I_2^2(K) + \frac{1}{K} \Bigr),
    \end{split}
  \end{equation*}
  where we used \eqref{eq:stopgronwall} to obtain the second inequality.  Gronwall's inequality  and then  Fatou's lemma  yield at last
  \begin{equation*}
    \EE \bigg( \bigg( \frac{N_{t}^{K}}{K}-\langle\mu_{t},1\rangle \bigg)^4 \bigg) \leq{}   I_4^4(K) +  C_T\biggl( \frac{1}{K} +  \frac{I_2^2(K)}{K^2}  \biggr), 
  \end{equation*}
  and we obtain the asserted bound noting that $I_2^2(K)\leq \sqrt{ I_4^4(K)}\leq  1+ I_4^4(K)$.
\end{proof}

We prove now Proposition \ref{prop:nonlinearprocess}, which relates the solution $(\mu_t)_{t\geq0}$ of equation \eqref{bdi:limiteq} to a nonlinear process of McKean-Vlasov type.

\begin{proof}[Proof of Proposition \ref{prop:nonlinearprocess}]
  Pathwise existence and uniqueness for the SDE \eqref{eq:nonlindiffusion} comes from the fact  that the coefficients are Lipschitz functions. In order to characterize the flow of time-marginal laws of $(Y_t)_{t\geq0}$,  consider a function $f \in C^{1,2}([0,T]\times\RR^d)$ satisfying the conditions in Theorem \ref{thm:largepopulationlimit}. By Itô's formula we obtain
  \begin{multline*}
    f(t,Y_t) = f(0,Y_0) + \int_0^t \frac{\partial f(s,Y_s)}{\partial s} \,\dd s +\int_0^t \nabla f(s,Y_s)^{\mathrm{t}}b(Y_s, H*\mu_s(Y_s)) \,\dd s \\
    + \int_0^t \nabla f(s,Y_s)^{\mathrm{t}}\sigma(Y_s, G*\mu_s(Y_s)) \,\dd W_s + \frac{1}{2}\int_0^t \mathrm{Tr}(a(Y_s,G*\mu_s(Y_s))\mathrm{Hess}f(s,Y_s)) \,\dd s.
  \end{multline*}
  Taking  expectation shows  that the law of the time-marginal is a weak solution of equation \eqref{bdi:normalizedeq} with respect to that set of test functions.  Now, consider the function $h(t,x)= \langle\mu_t,1\rangle f(t,x)$. By equation \eqref{bdi:normalizedeq} we get
  \begin{align*}
    \langle\bar{\mu}_t ,h(t, \cdot)\rangle ={}& \langle\bar{\mu}_0,h(0,\cdot)\rangle + \int_0^t \langle\bar{\mu}_s,\partial_sh(s,\cdot)+ L_{\mu_s}h(s,\cdot) \rangle \,\dd s\\
    ={}& \langle\langle\mu_0,1\rangle\bar{\mu}_0,f(0,\cdot)\rangle + \int_0^t\langle \bar{\mu}_s , f(s,\cdot)\partial_s\langle\mu_s,1\rangle + \langle\mu_s,1\rangle\partial_s f(s,\cdot) + \langle\mu_s,1\rangle L_{\mu_s}f(s,\cdot)\rangle \,\dd s\\
    ={}&  \langle\langle\mu_0,1\rangle\bar{\mu}_0,f(0,\cdot)\rangle + \int_0^t\langle \langle\mu_s,1\rangle\bar{\mu}_s , \partial_s f(s,\cdot) + L_{\mu_s}f(s,\cdot) + (r-c\langle\mu_s,1\rangle)f(s,\cdot)\rangle \,\dd s,
  \end{align*}
  which implies that $ (\tilde{\mu}_t)_{t\geq 0}  \coloneqq (\langle\mu_t,1\rangle\bar{\mu}_t )_{t\geq 0}$ satisfies the  following ``linearized'' version of equation \eqref{crossdiffeq}
  \[ \left\langle\tilde{\mu}_t,f(t,\cdot)\right\rangle=\left\langle\mu_0,f(0,\cdot)\right\rangle+\int_0^t\left\langle\tilde{\mu}_s,\partial_s f(s,\cdot)+L_{\mu_s}f(s,\cdot)+(r-c\langle\mu_s,1\rangle)f(s,\cdot)\right\rangle\,\dd s. \]
  With similar (indeed simpler) arguments as in the uniqueness part of Theorem \ref{thm:largepopulationlimit} (see \cite[Section 4]{FM2015}) one can show that uniqueness of weak solutions (with respect to the same class of test functions)   of this equation holds.  Since    $ (\tilde{\mu}_t )_{t\geq 0}=  (\mu_t )_{t\geq 0}$ also is a solution, we deduce that $\langle\mu_t,1\rangle\bar{\mu}_t = \mu_t$ for all $t\geq0$. 

  The previous identity yields  $\langle\mu_t,f\rangle = \EE(\langle\mu_t,1\rangle f(Y_t))$ for every bounded measurable $f$, and  the fact that $(\langle\mu_t,1\rangle)_{t\geq0}$ is the unique solution of equation \eqref{eq:logisticODE} is readily obtained by taking $f=1$ in Theorem \ref{thm:largepopulationlimit},  recalling also that the local Lipschitz character of the ODE's coefficient ensures  uniqueness for it.  
\end{proof}

The following propagation of moments result for the unique solution of equation \eqref{bdi:normalizedeq} will be needed. 
\begin{lem}\label{lem:boundsmom}
  For each $T>0$  and $q\geq 2$ there is a constant $C'_{T}>0$ such that  
  \[ \sup_{t\in [0,T]} M_q(\bar{\mu}_t) < C'_T (1+M_q(\bar{\mu}_0)). \]
\end{lem}
 
\begin{proof}
  We will use the fact that diffusion process $(Y_t)_{t\geq 0}$   studied in  Proposition \ref{prop:nonlinearprocess} satisfies $\EE( \|Y_t\|^q)=  M_q(\bar{\mu}_t)$.  Applying Itô's formula to $\|Y_t\|^q$ for $q\geq2$ yields 
  \begin{multline}\label{proof:itoformulaq}
    \|Y_t\|^q = \|Y_0\|^q + \int_0^t q\|Y_s\|^{q-2} Y_s^{\mathrm{t}} b(Y_s,H*\mu_s(Y_s)) \,\dd s + \int_0^t q\|Y_s\|^{q-2} Y_s^{\mathrm{t}} \sigma(Y_s,G*\mu_s(Y_s)) \,\dd B_s \\
    + \frac{1}{2}\sum_{i,j=1}^d\sum_{k=1}^d \int_0^t \Bigl(q(q-2)\|Y_s\|^{q-4}|Y_s^{(i)}||Y_s^{(j)}|+\delta_{ij}\|Y_s\|^{q-2}\Bigr) \\
    \times\sigma^{(ik)}(Y_s,G*\mu_s(Y_s))\sigma^{(jk)}(Y_s,G*\mu_s(Y_s))) \,\dd s.
  \end{multline}
  Since $b$ is Lipschitz we have $ \|b(Y_s,H*\mu_s(Y_s))\| \leq C\bigl(1+\|Y_s\|+|H*\mu_s(Y_s)|\bigr)$ 
  with  $ | H*\mu_s(Y_s)| = |\int H(x-Y_s)\mu_s(\dd x) | \leq \|H\|_\infty\sup_{t\in[0,T]}|\langle\mu_s,1\rangle|$ and similarly for $\sigma$ and $G$. We thus get that
  \[ \|b(Y_s,H*\mu_s(Y_s))\|\leq C\bigl(1+\|X_s\|\bigr) \text{ and } \|\sigma(X_s,G*\mu_s(X_s))\|\leq C\bigl(1+\|X_s\|\bigr).\]
  Using this in \eqref{proof:itoformulaq} gives us the bound
  \begin{multline*}
    \|Y_t\|^q \leq \|Y_0\|^q + C\int_0^t \|Y_s\|^{q-2} \,\dd s + C\int_0^t \|Y_s\|^{q-1} \,\dd s + C\int_0^t \|Y_s\|^{q} \,\dd s \\
    + \int_0^t q\|Y_s\|^{q-2} Y_s^{\mathrm{t}} \sigma(Y_s,G*\mu_s(Y_s)) \,\dd B_s.
  \end{multline*}
  Let now $(\tau_n)_{n\in\NN}$ be a localizing sequence for the local martingale in the right hand side. Taking expectation of the stopped process yields
  \begin{multline*}
    \EE(\|Y_{t\wedge\tau_n}\|^q) \leq \EE(\|Y_0\|^q) + C\int_0^{t} \EE(\|Y_{s\wedge\tau_n}\|^{q-2}) \,\dd s + C\int_0^{t} \EE(\|Y_{s\wedge\tau_n}\|^{q-1}) \,\dd s \\
    + C\int_0^{t} \EE(\|Y_{s\wedge\tau_n}\|^{q}) \,\dd s.
  \end{multline*}
  Notice that, by  Hölder's inequality,  one gets
  \[\int_0^{t}\EE(\|Y_{s\wedge\tau_n}\|^{q-1}) \,\dd s \leq \int_0^{t}\EE(\|Y_{s\wedge\tau_n}\|^{q})^{\frac{q-1}{q}} \,\dd s \leq C_T + C\int_0^{t}\EE(\|Y_{s\wedge\tau_n}\|^q) \,\dd s,\]
  and a similar bound holds for the term of order $q-2$. Combined with the previous, this entails
  \[\EE(\|Y_{t\wedge\tau_n}\|^q) \leq \EE(\|Y_0\|)^q + C_T + C\int_0^{t}\EE(\|Y_{s\wedge\tau_n}\|^q)\,\dd s,\]
  from where Gronwall's lemma yields
  \[ \EE(\|Y_{t\wedge\tau_n}\|^q) \leq C_T(\EE(\|Y_0\|)^q + 1). \]
 We conclude with  Fatou's lemma  taking $n\to\infty$.
\end{proof}

In order to check that  condition (\hyperlink{c3}{C.3}) holds, we need some additional bounds stated in the next two results (respectively analogous to Lemmas \ref{bd:lem:orig_indepdistance}  and \ref{bd:lem:lipschitzdyn} in the pure branching case).
In particular, the following result will be used to control the joint evolution of coupled particles in the two systems, between birth and death events.

\begin{lem}\label{bdi:lem:lipschitzdyn}
  Let $N$ and $K\in \NN\setminus \{0\}$ be deterministic and fixed, and consider the diffusion processes $(X^n)_{n=1}^N$ in $(\RR^d)^N$  evolving according to
  \begin{align*}
    \dd X_t^n = b(X_t^n, H*\mu_t^K(X_t^n)) \,\dd t + \sigma(X_t^n,G*\mu_t^K(X_t^n)) \,\dd B_t^n,\quad t\geq 0,
  \end{align*}
  where  $(B^n)_{n=1}^N$ are independent Brownian motions in $\RR^d$ and  $\mu_t^K$ stands for the empirical measure  $\mu_t^K=\frac{1}{K} \sum_{n=1}^N \delta_{X_t^n}$ of constant mass $N/K$. Consider also $N$ i.i.d. copies $(Y^n)_{n=1}^N$ of the process \eqref{eq:nonlindiffusion}, 
  \begin{align*}
    \dd Y_t^n = b(Y_t^n, H*\mu_t(Y_t^n)) \,\dd t + \sigma(Y_t^n,G*\mu_t(Y_t^n)) \,\dd B_t^n, \quad t\geq 0,
  \end{align*}
  driven by the same Brownian motions $(B^n)_{n=1}^N$ . For each $T>0$, there is  $C_T>0$ not depending on $K$ nor on $N$ such that  for all   $0<u<t<T$ and each $n=1,\dots,N$, 
  \begin{align*}
  \EE(\| X_t^n-Y_t^n\|^2-\| X_{u}^n-Y_{u}^n\|^2)\leq&\ C_T\int_{u}^t\EE(\| X_s^n-Y_s^n\|^2) \,\dd s + \int_u^t\EE\Bigl(\bigl\Vert\mu_s^K-\mu_s\bigr\Vert_{\BL^*}^2\Bigr) \,\dd s.
  \end{align*}
\end{lem}

\begin{proof}
  We first check that the running supremum of each  process   $(X^n)$ is square integrable.  Using similar bounds as  in the proof of Lemma \ref{lem:boundsmom}, we get for each $t\in[0,T]$, 
  \begin{equation*}
    \begin{split}
    \|X_t^n\|^2 &\leq \|X_0^n\|^2 + \int_0^t2 \|X_s^n\| \|b(X_s^n, H*\mu_s^K(X_s^n))\| \,\dd s  + \int_0^t 2(X_s^n)^{\mathrm{t}}\sigma(X_s^n,G*\mu_s^K(X_s^n)) \,\dd B_s \\
    &\qquad + \int_0^t \|\sigma(X_s^n,G*\mu_s^K(X_s^n))\|^2 \,\dd s \\
    &\leq \|X_0^n\|^2 + C_T + C\int_0^t \|X_s^n\| \,\dd s + C\int_0^t \|X_s^n\|^2 \,\dd s + C\int_0^t \|X_s^n\| |H*\mu_s^K(X_s^n)| \,\dd s \\
    &\qquad + C\int_0^t |G*\mu_s^K(X_s^n)|^2 \,\dd s + \int_0^t 2(X_s^n)^{\mathrm{t}}\sigma(X_s^n,G*\mu_s^K(X_s^n)) \,\dd B_s \\
    &\leq \|X_0^n\|^2 + C_T +  CT \|H\|_\infty^2\left(\frac{N}{K}\right)^2 + CT \|G\|_\infty^2\left(\frac{N}{K}\right)^2 + C\int_0^t \|X_s^n\|^2 \,\dd s \\
    &\qquad + \int_0^t 2(X_s^n)^{\mathrm{t}}\sigma(X_s^n,G*\mu_s^K(X_s^n)) \,\dd B_s,
    \end{split}
  \end{equation*}
  since, in the present lemma's setting, $ \langle\mu_s^K,1\rangle = N/K$ for all $s\geq0$.   Let $(\tau_m)_{m\in\NN}$ be a localizing sequence for the local martingale in the previous inequality. As in the proof of Lemma \ref{bd:lem:lipschitzdyn} we localize and then we take supremum until time $t\wedge\tau_m$ on both sides, obtaining that
  \begin{equation*}
    \begin{split}
    \sup_{u\in[0,t\wedge\tau_m]}\|X_u^n\|^2 &\leq \|X_0^n\|^2 + C_T +  CT \|H\|_\infty^2\biggl(\frac{N}{K}\biggr)^2 + CT \|G\|_\infty^2\biggl(\frac{N}{K}\biggr)^2 + C\int_0^t \sup_{u\in[0,s\wedge\tau_m]}\|X_u^n\|^2 \,\dd s \\
    &\qquad+ \sum_{i,j=1}^d\biggl(\sup_{u\in[0,t\wedge\tau_m]} \biggl|\int_0^u 2(X_s^n)^{(i)}\sigma^{(ij)}(X_s^n,G*\mu_s^K(X_s^n)) \,\dd B_s^{(j)}\biggr|\biggr).
    \end{split}
  \end{equation*}
  The expectation  of the last term is controlled using the BDG inequality by
  \begin{align*}
    \sum_{i,j=1}^d\EE\biggl(\sup_{u\in[0,t\wedge\tau_m]} \biggl|\int_0^u 2(X_s^n)^{(i)}\sigma^{(ij)}(X_s^n,G*\mu_s^K(X_s^n)) \,\dd B_s^{(j)}\biggr| \biggr) {}& \\
    {}&\hspace{-25.5ex} \leq \sum_{i,j=1}^d\EE\biggl( \biggl(\int_0^{t\wedge\tau_m} 4\Bigl((X_s^n)^{(i)}\sigma^{(ij)}(X_s^n,G*\mu_s^K(X_s^n))\Bigr)^2 \,\dd s\biggr)^{\frac{1}{2}} \biggr) \\
    {}&\hspace{-25.5ex} \leq C\EE\biggl( \biggl(\int_0^{t\wedge\tau_m} \|X_s^n\|^2 \|\sigma(X_s^n,G*\mu_s^K(X_s^n))\|^2 \,\dd s\biggr)^{\frac{1}{2}} \biggr) \\
    {}&\hspace{-25.5ex} \leq C\EE\biggl(\biggl( 1 + \|G\|_\infty^2\left(\frac{N}{K}\right)^2 \biggr)^{\frac{1}{2}} \biggl(\int_0^t \|X_{s\wedge\tau_m}^n\|^2 \,\dd s\biggr)^{\frac{1}{2}} \biggr) \\
    {}&\hspace{-25.5ex} \leq \biggl(1+\biggl(\frac{N}{K}\biggr)^2\biggr)\biggl(C_T + C_T\int_0^t \EE(\|X_{s\wedge\tau_m}^n\|^2) \,\dd s\biggr).
  \end{align*}
  This allows us to deduce that
  \begin{align*}
    \EE\biggl(\sup_{u\in[0,t\wedge\tau_m]}\|X_u^n\|^2\biggr) \leq \EE\bigl(\|X_0^n\|^2\bigr) + C_{T,N,K} + C_{T,N,K}\int_0^t \EE\biggl(\sup_{u\in[0,s\wedge\tau_m]}\|X_u^n\|^2 \biggr) \,\dd s,
  \end{align*}
  where $C_{T,N,K}$ is a constant depending on $T,N$ and $K$ (recalling that $N$ and $K$ are deterministic in the setting of this lemma). From this last inequality, Gronwall's lemma and monotone convergence when $m\to\infty$ yield
  \begin{equation*}
    \EE\biggl( \sup_{t\in[0,T]}\|X_t^n\|^2\biggr) < \infty.
  \end{equation*}
  A similar argument can be applied to the process $(Y_t^n)_{t\geq0}$ in order to obtain the same conclusion.
  We now apply Itô's formula for fixed $n$ to  get 
  \begin{multline*}
    \|X_t^n-Y_t^n\|^2 = \|X_u^n-Y_u^n\|^2+ \int_u^t2(X_s^n-Y_s^n)^{\mathrm{t}}\bigl(b(X_s^n,H*\mu_s^K(X_s^n)) - b(Y_s^n,H*\mu_s(Y_s^n))\bigr)\, \dd s\\
    \hspace{11ex} + \int_u^t 2(X_s^n-Y_s^n)^{\mathrm{t}} \bigl(\sigma(X_s^n,G*\mu_s^K(X_s^n))-\sigma(Y_s^n,G*\mu_s(Y_s^n))\bigr)\, \dd B_s^n\\
    +\sum_{i,j=1}^d\int_u^t\bigl(\sigma^{(ij)}(X_s^n,G*\mu_s^K(X_s^n))-\sigma^{(ij)}(Y_s^n,G*\mu_s(Y_s^n))\bigr)^2\, \dd s.
  \end{multline*}
  Using the Lipschitz character of the coefficients we get the bound
  \begin{multline*}
    \|X_t^n-Y_t^n\|^2 \leq \|X_u^n-Y_u^n\|^2 + C\int_u^t\bigl(\|X_s^n-Y_s^n\|^{2} + \|X_s^n-Y_s^n\||H*\mu_s^K(X_s^n)-H*\mu_s(Y_s^n)|\bigr)\,\dd s\\
    + C\int_u^v \bigl(\|X_s^n-Y_s^n\|^{2} + |G*\mu_s^K(X_s^n)-G*\mu_s(Y_s^n)|^2\bigr)\,\dd s\\
    + \int_u^t 2(X_s^n-Y_s^n)^{\mathrm{t}} (\sigma(X_s^n,G*\mu_s^K(X_s^n))-\sigma(Y_s^n,G*\mu_s(Y_s^n)))\, \dd B_s^n.
  \end{multline*}
  Recalling that the function $H(\cdot-x)$ is bounded and Lipschitz  for each $x\in\RR^d$, we see that
  \begin{align*}
    \left\vert H*\mu_s^K(X_s^n)-H*\mu_s(Y_s^n)\right\vert \leq&\ \left\vert H*\mu_s^K(X_s^n) - H*\mu_s(X_s^n)\right\vert+\left\vert H*\mu_s(X_s^n) - H*\mu_s(Y_s^n)\right\vert\\
    \leq&\ C\Vert\mu_s^K-\mu_s\Vert_{\BL^*}+C\Vert\mu_s\Vert_{\BL^*} \|X_s^n-Y_s^n\|,
  \end{align*}
  and similarly for the terms involving $G$. The uniform bound on the mass of  $(\mu_t)_{t\geq0}$ on finite time intervals allows us to get for all $0<u<t<T$ that
  \begin{equation*}
    \begin{split}
    \|X_t^n-Y_t^n\|^2 &\leq \|X_u^n-Y_u^n\|^2 + C\int_u^t \bigl(\|X_s^n-Y_s^n\|^{2} + \|X_s^n-Y_s^n\|\|\mu_s^K-\mu_s\|_{\BL^*}\bigr)\, \dd s\\
    &\qquad + C\int_u^v \bigl(\|X_s^n-Y_s^n\|^{2} + \|\mu_s^K-\mu_s\|_{\BL^*}^2\bigr)\, \dd s\\
    &\qquad\qquad + \int_u^t 2(X_s^n-Y_s^n)^{\mathrm{t}} (\sigma(X_s^n,G*\mu_s^K(X_s^n))-\sigma(Y_s^n,G*\mu_s(Y_s^n)))\, \dd B_s^n \\
    &\leq \|X_u^n-Y_u^n\|^2 + C\int_u^t\bigl(\|X_s^n-Y_s^n\|^{2} + \|\mu_s^K-\mu_s\|_{\BL^*}^2\bigr)\, \dd s\\
    &\qquad + \int_u^t 2(X_s^n-Y_s^n)^{\mathrm{t}} (\sigma(X_s^n,G*\mu_s^K(X_s^n))-\sigma(Y_s^n,G*\mu_s(Y_s^n)))\, \dd B_s^n,
    \end{split}
  \end{equation*}
  where we used Young's inequality for the second inequality, and where $C$ is a constant not depending on $K$ nor on $N$ that changed from line to line.  By considering a localizing sequence $(\tau_m)_{m}$ for the local martingale on the right hand side, we can take expectation  of the stopped processes to obtain
  \begin{multline*}
    \EE(\| X_{t\wedge\tau_m}^n-Y_{t\wedge\tau_m}^n\|^2 ) \leq \EE(\| X_{u}^n-Y_{u}^n\|^2) + C\int_{u}^{t}\EE(\| X_{s\wedge\tau_m}^n-Y_{s\wedge\tau_m}^n\|^2)\,\dd s \\
    + \int_u^{t}\EE\bigl(\bigl\Vert\mu_{s\wedge\tau_m}^K-\mu_{s\wedge\tau_m}\bigr\Vert_{\BL^*}^2\bigr)\,\dd s,
  \end{multline*} 
  for all $0<u<t<T$.  Thanks to  the second moments controls on the running suprema  of $X^n$ and $Y^n$,  and since the total  mass of $\mu^K_t$is constant in the context of the present lemma, we can use dominated convergence to take $m\to \infty$ and conclude the proof.  
\end{proof}

The following bound gathering  all the previous estimates will allow us to check that condition (\hyperlink{c3}{C.3}) holds.

\begin{lem}\label{lem:systemsdistance}
  For $t\in[0,T]$
  \begin{align*}
    \EE\Big(\frac{1}{K}\sum_{n=1}^{N_t^{K}}\| X_t^n-Y_t^n\|^2\Big) \leq{}& C_T\biggl[ I_4^2(K) + K^{-\frac{1}{2}} + \int_0^T\EE\Big(\frac{N_s^{K}}{K}W_2^2(\bar{\nu}_s^K,\bar{\mu}_s)\Big)\,\dd s\biggr].
  \end{align*}
  where $C_T>0$ is a constant that depends on the parameters of the model.
\end{lem}

\begin{proof}
  As in  the proof of Lemma \ref{bd:lem:orig_indepdistance} we consider the  product empirical measure $\eta_t^K\coloneqq\frac{1}{K}\sum_{n=1}^{N_t^{K}}\delta_{(X_t^n,Y_t^n)}$ and  decompose again
  \[\EE\biggl(\frac{1}{K}\sum_{n=1}^{N_t^{K}}\vert X_t^n-Y_t^n\vert^2\biggr)=\EE(\langle\eta_t^K,d_2\rangle),\]
  in terms of  the sequence of jump times $(T_m)_{m\in\mathbb{N}}$,  as in \eqref{proof:jumpdecomp}. We can proceed in a similar  way as in \eqref{proof:betweenjumps} to control the evolution between jumps, now with help of  Lemma \ref{bdi:lem:lipschitzdyn}, and control the contributions in the jump instants in the same way as in \eqref{proof:jumps}, to obtain
  \begin{multline*}
    \EE(\langle\eta_t^K,d_2\rangle) \leq C\int_0^t\EE(\langle\eta_s^K,d_2\rangle)\,\dd s+C\int_0^t\EE\Bigl(\frac{N_s^{K}}{K}W_2^2(\bar{\nu}_s^K,\bar{\mu}_s)\Bigr) \,\dd s \\
    +C\int_0^t\EE\Bigl(\frac{N_s^{K}}{K}\Vert\mu_s^K-\mu_s\Vert_{\BL^*}^2\Bigr) \,\dd s,
  \end{multline*}
  where $C$ is a positive constant. Thus, with respect to the case dealt with in the previous section,  incorporating interactions at the level of the dynamics only  results in the addition of the last term. In order to  bound this new term,  we use Lemma \ref{lem:normwasserstein} to get
  \begin{align*}
    \EE\Big(\frac{N_s^{K}}{K}\bigl\Vert\mu_s^K-\mu_s\bigr\Vert_{\BL^*}^2\Big)&\leq\EE\Big(\frac{N_s^{K}}{K}\Big(\langle\mu_s,1\rangle\Vert\bar{\mu}_s^K-\bar{\mu}_s\Vert_{\BL^*}+\Big\vert\frac{N_s^{K}}{K}-\langle\mu_s,1\rangle\Big\vert\Big)^2\Big) \notag\\
    &\leq 2\sup_{u\in[0,T]}\langle\mu_u,1\rangle^2\EE\Big(\frac{N_s^{K}}{K}\Vert\bar{\mu}_s^K-\bar{\mu}_s\Vert_{\BL^*}^2\Big)+2\EE\Big(\frac{N_s^{K}}{K}\Big|\frac{N_s^{K}}{K}-\langle\mu_s,1\rangle\Big|^2\Big) \notag\\
    &\leq C\EE\Big(\frac{N_s^{K}}{K}\Vert\bar{\mu}_s^K-\bar{\nu}_s^K\Vert_{\BL^*}^2\Big)+C\EE\Big(\frac{N_s^{K}}{K}\Vert\bar{\mu}_s-\bar{\nu}_s^K\Vert_{\BL^*}^2\Big) \notag\\
    &\qquad +2\EE\Big(\frac{N_s^{K}}{K}\Big|\frac{N_s^{K}}{K}-\langle\mu_s,1\rangle\Big|^2\Big),
  \end{align*}
  where the control on the mass of the solution to equation \eqref{bdi:limiteq} on finite time intervals is used. To control the first term of the right hand side,  we relate it to the Wasserstein distance using again Lemma \ref{lem:normwasserstein}, obtaining 
  \[\EE\Big(\frac{N_s^{K}}{K}\Vert\bar{\mu}^K_s-\bar{\nu}_s^K\Vert_{\BL^*}^2\Big)\leq\EE \Big(\frac{N_s^{K}}{K}W_2^2(\bar{\mu}_s^K,\bar{\nu}_s^K)\Big)\leq\EE(\langle\eta_s^K,d_2\rangle).\]
  We do the same with the second term to get
  \[\EE\Big(\frac{N_s^{K}}{K}\Vert\bar{\mu}_s-\bar{\nu}_s^K\Vert_{\BL^*}^2\Big)\leq\EE \Big(\frac{N_s^{K}}{K}W_2^2(\bar{\mu}_s,\bar{\nu}_s^K)\Big).\]
  We thus obtain the inequality
  \begin{multline*}
    \EE(\langle\eta_t^K,d_2\rangle) \leq  C\int_0^t\EE(\langle\eta_s^K,d_2\rangle)\,\dd s+C\int_0^t\EE\Big(\frac{N_s^{K}}{K}W_2^2(\bar{\nu}_s^K,\bar{\mu}_s)\Big)\,\dd s \\
    + 2\int_0^t\EE\Big(\frac{N_s^{K}}{K}\Big|\frac{N_s^{K}}{K}-\langle\mu_s,1\rangle\Big|^2\Big)\,\dd s, 
  \end{multline*}
  where only the last term needs to be controlled. Using Hölder's inequality yields
  \begin{align*}
    \EE\Big(\frac{N_s^{K}}{K}\Big|\frac{N_s^{K}}{K}-\langle\mu_s,1\rangle\Big|^2\Big)\leq&\EE\Big(\Big(\frac{N_s^{K}}{K}\Big)^2\Big)^{\frac{1}{2}}\EE\Big(\Big|\frac{N_s^{K}}{K}-\langle\mu_s,1\rangle\Big|^4\Big)^{\frac{1}{2}},
  \end{align*}
  where the first factor on the r.h.s. is controlled by Lemma \ref{lem:boundsmass}. Thanks to the second bound in Lemma \ref{lem:boundsmass}, we obtain that
  \begin{align*}
    \EE(\langle\eta_t^K,d_2\rangle)\leq&\ C\int_0^t\EE(\langle\eta_s^K,d_2\rangle)\,\dd s+C\int_0^t\EE\Big(\frac{N_s^{K}}{K}W_2^2(\bar{\nu}_s^K,\bar{\mu}_s)\Big)\,\dd s  +C_T\biggl( I_4^2(K) + \frac{1}{\sqrt{K}} \biggr). 
  \end{align*}
  Finally, Gronwall's lemma  yields 
  \[ \EE(\langle\eta_t^K,d_2\rangle)\leq C_T\bigg[ I_4^2(K) + \frac{1}{\sqrt{K}} + \int_0^T\EE\Big(\frac{N_s^{K}}{K}W_2^2(\bar{\nu}_s^K,\bar{\mu}_s)\Big)\,\dd s\bigg]e^{CT}. \]
\end{proof}

We deduce the following result. 

\begin{cor}
  Condition \textnormal{(\hyperlink{c3}{C.3})} holds.
\end{cor}

\begin{proof}
  Applying Lemma \ref{lem:systemsdistance}, Lemma \ref{lem:boundNW} and noting  that $  1/\sqrt{K}\leq  C  R_{d,q}(K)$, we obtain the bound
  \begin{equation}\label{eq:boundforC3}  
    \EE(\langle\eta_t^K,d_2\rangle) \leq  C_T \Bigl(I_4^2(K) + R_{d,q}(K)\Bigr). 
  \end{equation}
  It suffices to combine this with the inequality  $  \EE\bigl(\frac{N_t^{K}}{K}W_2^2\bigl(\bar{\mu}_t^K,\bar{\nu}_t^K\bigr)\bigr)\leq\EE\bigl(\frac{1}{K}\sum_{n=1}^{N_t^{K}}\left\| X_t^n-Y_t^n\right\|^2\bigr)$.
\end{proof}

Finally, everything is in place to prove the main result.

\begin{proof}[Proof of Theorem \ref{thm:rateofconvergence} under \textnormal{{(\hyperlink{h}{H})}}]
  Following \eqref{eq:boundproofthm2} and using Lemma \ref{lem:boundsmass}, we obtain
  \begin{equation*}
  \begin{split}
    \EE\left(  \Vert\mu_t^K-\mu_t\Vert_ {{\BL}^*}  \right) &\leq \biggl( \EE\Big(\frac{N_t^{K}}{K}W_2^2\bigl(\bar{\nu}_t^K,\bar{\mu}_t\bigr)\Big)^\frac{1}{2} +  \EE\Big(\frac{N_t^{K}}{K}W_2^2\bigl(\bar{\nu}_t^K,\bar{\mu}_t^K\bigr)\Big)^\frac{1}{2}  \biggr)  \EE\Big(\frac{N_t^{K}}{K} \Big)^\frac{1}{2} \\
    &\qquad +  \EE \Bigl( \bigl( \langle\mu^K_t,1\rangle - \langle\mu_t,1\rangle \bigr)^2 \Bigr)^{\frac{1}{2}} \\
    &\leq C_T  \biggl( \EE\Big(\frac{N_t^{K}}{K}W_2^2\bigl(\bar{\nu}_t^K,\bar{\mu}_t\bigr)\Big)^\frac{1}{2} +  \EE\Big(\frac{N_t^{K}}{K}W_2^2\bigl(\bar{\nu}_t^K,\bar{\mu}_t^K\bigr)\Big)^\frac{1}{2} +  I_2(K) +K^{-1/2}  \biggr).
  \end{split}
  \end{equation*}
  As in the previous section, thanks to condition (\hyperlink{c}{C}), Lemma \ref{lem:boundsmass}, Lemma \ref{lem:boundsmom}, and Lemma \ref{lem:boundNW} we obtain 
  \[
    \EE\left(  \Vert\mu_t^K-\mu_t\Vert_ {{\BL}^*}  \right) \leq C_T \Bigl( R_{d,q}(K)^{\frac{1}{2}} + I_4(K) \Bigr),
  \]
  since $I_2(K)\leq I_4(K)$, concluding thus the proof.
\end{proof}

We end this section proving the conditional propagation of chaos property stated in  Corollary  \ref{cor:condiquantipropchaos}. 

\begin{proof}[Proof of Corollary \ref{cor:condiquantipropchaos}]
Let  $ \Psi_{d,q}(K) $ denote the function of $K$ appearing on the right hand side of the bound in Theorem \ref{thm:rateofconvergence}.     By exchangeability of  $\bigl( (X_t^1,Y_t^1), \dots,  \bigl(X_t^{\mbox{\scalebox{0.7}{$N^K_t$}}},Y_t^{\mbox{\scalebox{0.7}{$N^K_t$}}}\bigr)\bigr)$  conditionally on $N_t^K$, for all $t\geq 0$ we get  \begin{equation}\label{eq:NKM2}
    \EE\Bigl(\frac{N_t^K}{K} \| X_t^1- Y_t^1\|^2 \Bigr) =   \EE\biggl(\frac{1}{K} \sum_{n=1}^{N_t^K} \| X_t^n- Y_t^n\|^2 \biggr)\leq C_t  \Psi^2_{d,q}(K) , 
  \end{equation}
  thanks to \eqref{eq:boundforC3}.  By Proposition \ref{prop:condindep}, we have $ {\mathcal{L}} \big( Y_t^{1},\dots, Y_t^{j}  \mid N_t^K\big) =  \bar{\mu}_t^{\otimes j}$ on the event $\{j\leq N_t^K\}$. Now, letting $c_t:= \langle  \mu_t , 1\rangle \in (0,\infty) $ denote the limit in law of $N_t^K/K$,  and using the second inequality of  Lemma \ref{lem:normwasserstein} in the third bound below   we get,  for all $\varepsilon >0$, that
  \begin{equation*}
  \begin{split}
    \PP\Bigl( \Bigl\| {\mathcal{L}} \Big( X_t^{1},\dots, X_t^{j\wedge N_t^K}  &\Bigm\vert N_t^K \Big) - \bar{\mu}_t^{\otimes j}\Bigr\|_{\BL^*} > \varepsilon, \, N_t^K\geq j \Bigr) \\ 
    & \leq \PP\biggl( \frac{N_t^K}{K}\Bigl\| {\mathcal{L}} \Big( X_t^{1},\dots, X_t^{j}  \Bigm\vert N_t^K\Big) - \bar{\mu}_t^{\otimes j}\Bigr\|_{\BL^*} \biggl(\frac{N_t^K}{K}\biggr)^{-1} > \frac{\varepsilon c_t}{2} \frac{2}{c_t},  \, N_t^K\geq j  \biggr) \\
    & \leq \PP\biggl( \frac{N_t^K}{K}\Bigl\| {\mathcal{L}} \Big( X_t^{1},\dots, X_t^{j}  \Bigm\vert N_t^K\Big) - \bar{\mu}_t^{\otimes j}\Bigr\|_{\BL^*} > \frac{\varepsilon c_t}{2} , \,  N_t^K\geq j    \biggr) \\
    &\qquad + \PP\biggl( \frac{N_t^K}{K}< \frac{c_t}{2}\biggr) \\
    & \leq \frac{2}{\varepsilon c_t} \EE\biggl(\frac{N_t^K}{K}\EE\biggl(\sum_{n=1}^{j} \| X_t^n- Y_t^n\| \biggm\vert  N_t^K \biggr) \1_{\bigl\{N_t^K \geq j\bigr\}}\biggr) + \PP\biggl( \frac{N_t^K}{K}< \frac{c_t}{2} \biggr) \\
    & \leq   \frac{2j}{\varepsilon c_t} \EE\Bigl(\frac{N_t^K}{K} \| X_t^1- Y_t^1\| \Bigr) + \PP\biggl( \frac{N_t^K}{K} < \frac{c_t}{2} \biggr) \\
    & \leq  \frac{2j}{\varepsilon c_t} C_t'   \Psi_{d,q}^2(K) + \PP\biggl( \frac{N_t^K}{K} < \frac{c_t}{2} \biggr),
  \end{split}
  \end{equation*}
  using also the Cauchy-Schwarz inequality, the estimate  \eqref{eq:NKM2}  and the fact that $\EE( N_t^K/K)^{1/2} <\infty $  in  the last inequality. Since $N_t^K/K\to c_t$ in law,  the terms in the last line go to $0$ when  $K\to \infty$. The  convergence $ \PP( N_t^K\geq j)\to 1$ then yields 
  \[ \PP\Bigl( \Bigl\| {\mathcal{L}} \Big( X_t^{1},\dots, X_t^{j\wedge N_t^K}  \Bigm\vert N_t^K \Big) - \bar{\mu}_t^{\otimes j}\Bigr\|_{\BL^*} > \varepsilon \Bigm\vert \, N_t^K\geq j \Bigr)  \longrightarrow 0 \]
  as $K\to \infty$ and the statement follows.
\end{proof}

\section{Extensions}\label{sec:extensions}

We finish with some remarks regarding possible extensions of our approach, and the technical issues that must be solved in order to establish similar results in some related, more general settings. 

\begin{rmk}\label{rmk:initialcondcoupling}
If instead of \textnormal{(\hyperlink{h1}{H.1})} it is assumed that the initial data $\mu^K_0$ satisfies the condition  in Lemma \ref{lem:weakconv_mu0} \textnormal{b)},  the arguments and construction  leading to the proof of  Theorem \ref{thm:rateofconvergence} must be modified,  along the following lines: 
  \begin{itemize}
    \item  In condition \textnormal{(\hyperlink{c1}{C.1})},  $ \nu_0^K = \mu_0^K  $ is not enforced, but   $K \langle \nu_t^K, 1\rangle =K  \langle \mu_t^K, 1\rangle = N_ t^K$  is kept. 
    \item  In the construction of the coupling using algorithm \textnormal{(\hyperlink{a}{A})}, the random variables $(Y^k)_{k\geq 1}$ are chosen as before while,  for any $K$ and $N$, the random vectors  $(X_0^1,\dots, X_0 ^N) $ are  chosen  on the event $\{N_0^K=N\}$,  suitably coupled with $(Y_0^1,\dots, Y_0 ^N)$.  This results in an extra term of the form $\EE(\langle\eta_0^K,d_2\rangle)$ on the r.h.s. of the bounds in the  statement and  proof of Lemma \ref{lem:systemsdistance} which in turn translates into an additional term $C_T \EE(\langle\eta_0^K,d_2\rangle)^{1/2} $ on the r.h.s.  of the bound  in  Theorem \ref{thm:rateofconvergence}. 
    \item   In order to minimize the value of this additional term, the coupling of the variables $(X_0^1,\dots, X_0 ^N) $  and $(Y_0^1,\dots, Y_0 ^N)$ must be chosen on each event $\{N_0^K=N\}$ so as to realize the squared $2$-Wasserstein distance between the laws of  $(X_0^1,\dots, X_0 ^N) $  and $\bar{\mu}_0^{\otimes N}$ in $(\RR^d)^N$.  Denoting
      \[\widetilde{W}_2^2 ({\mathcal{L}} (X_0^1,\dots, X_0 ^N), \bar{\mu}_0^{\otimes N})  =\frac{1}{N} W_2^2 ({\mathcal{L}} (X_0^1,\dots, X_0 ^N), \bar{\mu}_0^{\otimes N}), \]
    the normalized squared $2$-Wasserstein distance, the additional term $   \EE(\langle\eta_0^K,d_2\rangle)^{1/2} $   then writes  
    \[
      \EE \biggl( \frac{N_0^K}{K} \widetilde{W}_2^2 \Bigl({\mathcal{L}} \bigl(X_0^1,\dots, X_0^{N_0^K} \bigm\vert  N_0^K\bigr), \bar{\mu}_0^{\otimes N_0^K} \Bigr) \biggr)^{1/2}.
    \]
\end{itemize}
\end{rmk}

 The ideas and techniques developed in this work can in principle also be extended to more general systems of interacting branching populations, including the general setting of  \cite{FM2015}. Nevertheless, this requires to deal  with significant additional technicalities,  and we have chosen to focus here on the basic ideas. The following possible  generalizations  are left for future work: 

\begin{itemize} 
  \item  The case of populations with spatially  or density depending birth or death events, as in the more general setting studied in \cite{FM2015},  seems feasible but presents one major additional difficulty, namely  that the jump times are correlated with the spatial dynamics.  The main consequence of this is that, in any  coupling with some auxiliary system of conditionally independent  (or less dependent) particles, the jump times cannot be expected to happen simultaneously. However, under the condition of spatial Lipschitz continuity of the reproduction rate and the competition kernel, it should be possible to  keep at least some subsystems effectively coupled on finite time intervals, while controlling explicitly  the discrepancy between jump times  in the two systems, in terms of the distance of the empirical measures of the systems themselves, in such a way that the discrepancies asymptotically vanish as the population size goes to infinity.  
  \item A further desirable generalization regards the case of branching events more general than binary ones.  The natural extension of the argument used here would consist in coupling all the offspring of a  branching particle in the original system,  with a set of equally many independent new  particles given birth at the same time in the auxiliary system. However it is not clear how to make compatible the use of optimal transport plans  to couple the branching particle  and the positions of the new particles in the auxiliary system,  with the independence requirement in the auxiliary system. A possible way of coping with  this problem could be to make a two-steps coupling  construction:  first, between the branching  particle in the original system and the positions of new particles in the auxiliary system (which would define an exchangeable  random vector of particles in any case) and, in a second step,  coupling those positions with independent particles with the required law. 
\end{itemize}

\appendix

\setcounter{secnumdepth}{0}

\hypertarget{appendix}{}
\section{Appendix}

\begin{proof}[Proof of Lemma \ref{lem:normwasserstein}]
  Since  $ \Vert \bar{\nu}\Vert_ {{\BL}^*} = \langle\bar{\nu},1\rangle=1$, we have 
  \begin{align*}
    \Vert\mu-\nu\Vert_ {{\BL}^*}=&\ \left\Vert \langle\mu,1\rangle\left(\bar{\mu}-\bar{\nu}\right)+\bar{\nu}\left(\langle\mu,1\rangle-\langle\nu,1\rangle\right)\right\Vert_ {{\BL}^*}\\
    \leq &\  \langle\mu,1\rangle \Vert\bar{\mu}-\bar{\nu}\Vert_ {{\BL}^*}+\big\vert\langle\mu,1\rangle-\langle\nu,1\rangle\big\vert. 
  \end{align*}
  Now, for any $\mu, \nu  \in {\mathcal{P}}(\RR^d)$, $\Vert\mu-\nu\Vert_ {{\BL}^*}=\sup_{\Vert\varphi\Vert_\BL\leq1}\left\vert\int_{\mathbb{R}^d\times\mathbb{R}^d}(\varphi(x)-\varphi(y))\,\pi(\dd x,\dd y)\right\vert $   for all coupling $\pi\in  {\mathcal{P}}(\RR^{2d}) $ of $\mu$ and $\nu$. Using the fact that   $\vert\varphi(x)-\varphi(y)\vert\leq\vert x-y\vert\wedge2$ when $\Vert\varphi\Vert_{\BL}\leq1$ and taking infimum over all $\pi\in\Pi(\mu,\nu)$ we conclude that $\Vert\mu-\nu\Vert_ {{\BL}^*}\leq \inf_{\pi\in\Pi(\mu,\nu)}\int \vert x-y\vert\wedge2\,\pi(\dd x,\dd y)\leq W_1(\mu,\nu).$
\end{proof}

\begin{proof}[Proof of Lemma \ref{lem:boundNW}]
  Write  $\alpha = 1/2 $ when  $d<4$ or  $\alpha = {2}/{d} $ when  $d>4$. Thanks to Theorem \ref{thm:FournierGuillin}, for some $C_{d,q}>0$, 
  \begin{equation*}
  \begin{split}
    \EE\Big(\frac{N }{K}W_2^2\bigl(\bar{\nu}^K,\bar{\mu}\bigr)\Big)= {}& \EE\Big(\frac{N}{K}\EE\Bigl(W_2^2\bigl(\bar{\nu}^K,\bar{\mu}\bigr)\Bigm| N \Bigr)\Big) \\
    \leq{}& C_{d,q} M_q^\frac{2}{q}(\bar{\mu})\,  \EE\Big(\frac{N}{K} \Big(N^{-\alpha }+N^{-  \frac{q-2}{q} }\Big)\Big) \\
    = {}& C_{d,q} M_q^\frac{2}{q}(\bar{\mu}) \Big( K^{-\alpha} \,  \EE \Big( \Big( \frac{N}{K}\Big)^{1-\alpha}\Big) + K^{- \frac{q-2}{q}} \EE  \Big( \Big(\frac{N}{K}\Big)^{ \frac{2}{q}}\Big)\Big) \\
    \leq{}& C_{d,q} M_q^\frac{2}{q}(\bar{\mu}) \Big( K^{-\alpha}  \EE  \Big( \frac{N}{K}\Big)^{1-\alpha}  + K^{- \frac{q-2}{q}} \EE   \Big(\frac{N}{K}\Big)^{ \frac{2}{q}}\Big),
  \end{split}
  \end{equation*}
  using Jensen's inequality in the last line.  This implies the result for $d\neq 4$.  When  $d=4$ we get the bounds
  \begin{equation*}
  \begin{split}
    \EE\Big(\frac{N }{K}W_2^2\bigl(\bar{\nu}^K,\bar{\mu}\bigr)\Big) \leq{}& C_{d,q} M_q^\frac{2}{q}(\bar{\mu}) \Big(  K^{-\frac{1}{2}} \, \EE\biggl(\Bigl(\frac{N}{K}\Bigr)^\frac{1}{2} \log(1+N) \biggr)     + K^{- \frac{q-2}{q}} \EE   \Big(\frac{N}{K}\Big)^{ \frac{2}{q}}\Big) \\
    \leq{}& C_{d,q}  M_q^\frac{2}{q}(\bar{\mu}) \Big(  K^{-\frac{1}{2}}  \, \EE \Bigl(\frac{N}{K}\Bigr)^\frac{1}{2} \EE \Bigl(\log^2(e+N) \Bigr)^\frac{1}{2}    + K^{- \frac{q-2}{q}} \EE   \Big(\frac{N}{K}\Big)^{ \frac{2}{q}}\Big).  \\
  \end{split}
  \end{equation*}
  The function  $x \in [e,\infty) \mapsto \log^2(x)$ being concave, we can extend it  linearly on  $(-\infty, e)$ to get a $C^1$ concave  function on $\RR$. Jensen's inequality then yields
  \[  \EE \Bigl(\log^2(e+N) \Bigr)^\frac{1}{2}    \leq  \log\Bigl(e+ K\EE\Bigl(\frac{N}{K}\Bigr)\Bigr)      \leq    1+\log(1+K)  + \log\Bigl(1\vee \EE\Bigl(\frac{N}{K}\Bigr)\Bigr). \]
  Using this, we finally obtain that
  \begin{equation*}
  \begin{split}
    \EE\Big(\frac{N }{K}W_2^2\bigl(\bar{\nu}^K,\bar{\mu}\bigr)\Big) &\leq C_{d,q}  M_q^\frac{2}{q}(\bar{\mu}) \Big( K^{-\frac{1}{2}}\, \EE \Bigl(\frac{N}{K}\Bigr)^\frac{1}{2}  + K^{-\frac{1}{2}}\log(1+K)  \, \EE \Bigl(\frac{N}{K}\Bigr)^\frac{1}{2} \\ 
    &\qquad + K^{-\frac{1}{2}}\, \EE \Bigl(\frac{N}{K}\Bigr)^\frac{1}{2}\log\Bigl(1\vee \EE\Bigl(\frac{N}{K}\Bigr)\Bigr) +  K^{- \frac{q-2}{q}} \EE   \Big(\frac{N}{K}\Big)^{ \frac{2}{q}}\Big),
  \end{split}
  \end{equation*}
  and the case  $d=4$ follows since $K^{-\frac{1}{2}}\leq K^{-\frac{1}{2}}\log(1+K)$ for $K\in \NN\setminus \{0\}$. 
\end{proof}

\begin{proof}[Proof of Lemma \ref{lem:weakconv_mu0}]
  Since condition (\hyperlink{h1}{H.1}) assumed in a) is a particular case of the assumptions in b),  it is enough to prove b) to  get both parts.   Taking $\mu=\mu_0$ and $\nu=\mu_0^K$ in  Lemma \ref{lem:normwasserstein},  we get 
  \begin{equation}\label{eq:limsupmu0K}
    \limsup_K  \PP(\Vert  \mu_0- \mu_0^K \Vert_{{\BL}^*} \geq \varepsilon) \leq  \limsup_K  \PP(\Vert  \bar{\mu}_0- \bar{\mu}_0^K \Vert_ {{\BL}^*} \geq \varepsilon/(2   \langle \mu_0, 1\rangle)  ),
  \end{equation}
  with $ \bar{\mu}_0^K = \frac{1}{N_0^K}\sum_{i=1}^{N_0^K} \delta_{X_0^i} $. On the other hand, for each $\delta >0$ and $M>0$,
  \begin{equation*}
  \begin{split}
    \PP(\Vert  \bar{\mu}_0- \bar{\mu}_0^K \Vert_ {{\BL}^*} \geq  \delta)  & \leq \sum_{N\geq M} \EE\left[  \PP (  \Vert  \bar{\mu}_0- \bar{\mu}_0^K \Vert_ {{\BL}^*} \geq  \delta \vert N_0^K=N) \1_{N_0^K=N} \right] + \PP( N_0^K< M)  \\
    &\leq \sup_{N\geq M}    \PP \left(  \bigg\Vert  \bar{\mu}_0-  \frac{1}{N}\sum_{i=1}^N \delta_{Y^{i,N}}  \bigg\Vert_ {{\BL}^*} \geq  \delta\right) + \PP(  \langle \mu_0^K, 1\rangle   < M/K ).   \\
  \end{split}
  \end{equation*}
  Since $ \langle \mu_0^K, 1\rangle$ converges weakly to a non null constant, the last term goes to $0$ when $K\to \infty$. On the other hand, it is well known that the  assumed $\bar{\mu}_0$-chaoticity is equivalent to the convergence in distribution of the random probability $ \frac{1}{N}\sum_{i=1}^N \delta_{Y^{i,N}} $ to $\bar{\mu}_0$ as $N\to \infty$.  If follows that  $\limsup_{K\to \infty}  \PP(\Vert  \bar{\mu}_0- \bar{\mu}_0^K \Vert \geq  \delta) =0$  which entails the  claim in view of \eqref{eq:limsupmu0K}. 
  
  c) The r.v.  $N_0^K=K  \langle \mu_0^K, 1\rangle$  is Poisson  of parameter $K \langle \nu_0,1\rangle $ and equals in law the sum $\sum_{i=1}^{ K }  N^i $ of independent  Poisson r.v.   $(N^i)_{i=1}^{ K  } $ of parameter $\langle \nu_0,1\rangle $.  By the law of large numbers, $ \langle \mu_0^K,1\rangle  = N_0^K /K $ converges in law to the constant  $ \langle \nu_0,1\rangle $. It is immediate from  basic properties of Poisson point measures that  the $N_0^K$  atoms of $  \mu_0^K$ are  i.i.d. of law $\bar{\nu}_0$ given   $\langle \mu_0^K, 1\rangle$, and we necessarily  have $\mu_0= \langle \nu_0,1\rangle \bar{\nu}_0 =\nu_0$.   Last,  $N_0^K$ being Poisson of parameter $K \langle \mu_0,1\rangle $, we have  $I_4^4(K) = K^{-3} \left( \langle \mu_0,1\rangle  + 3K  \langle \mu_0,1\rangle ^2\right) \leq C K^{-2}.$  \linebreak
\end{proof}

\paragraph*{Acknowledgments}
J.F. acknowledges partial support from Fondecyt Grant 1201948 and BASAL Fund  AFB170001 Center for Mathematical Modeling from ANID-Chile.  F. M.-H. acknowledges financial support received under the Doctoral Fellowship ANID-PFCHA/Doctorado Nacional/\allowbreak2017-21171912.  Both authors also thank support from Millennium Nucleus  Stochastic Models of Complex and Disordered Systems from Millennium Scientific Initiative. 

\bibliographystyle{plain}
\bibliography{QuantitativeLargePopulation}

\end{document}